\documentclass[12pt]{amsart}
\usepackage{bm}
\usepackage{tikz-cd}
\usepackage{amsmath}
\usepackage{fourier}
\usepackage{amssymb}
\usepackage{amscd}
\usepackage{amsthm}
\usepackage[centertags]{amsmath}
\usepackage{amsfonts}
\usepackage{newlfont}
\usepackage{graphicx}
\usepackage{amsfonts, amssymb}
\usepackage{mathrsfs}
\usepackage{latexsym}
\usepackage{tikz}
\usepackage{verbatim}
\usepackage[all]{xy}
\usepackage{enumitem}

\usepackage{setspace}

\usepackage[colorlinks=true,linkcolor=colorref,citecolor=colorcita,urlcolor=colorweb]{hyperref}
\definecolor{colorcita}{RGB}{21,86,130}
\definecolor{colorref}{RGB}{5,10,177}
\definecolor{colorweb}{RGB}{177,6,38}

\numberwithin{subsection}{section}

\newtheorem{theorem}{Theorem}[section]

\newtheorem{proposition}[theorem]{Proposition}
\newtheorem{corollary}[theorem]{Corollary}
\newtheorem{lemma}[theorem]{Lemma}

\theoremstyle{definition}
\newtheorem{remark}[theorem]{Remark}

\theoremstyle{remark}

\DeclareMathOperator{\id}{\mathrm{id}}

\newcommand{\sid}{\operatorname{\mathbf{Sid}}}

\DeclareMathOperator{\mon}{mon}

\newcommand{\C}{\mathbb{C}}

\newcommand{\TT}{\mathbb{T}}

\newcommand{\chimon}{\chi_{\mon}}

\definecolor{miverde}{RGB}{42, 180, 59}
\definecolor{tincho}{RGB}{128, 0, 255}

\begin{document}
\title[Support-Sensitive BH Inequalities and Local Invariants on Hamming Schemes]
{Support-Sensitive Bohnenblust-Hille Inequalities and Local Invariants on Hamming Schemes}

 \author[Defant]{A.~Defant}
 \address{%
 Institut f\"{u}r Mathematik,
 Carl von Ossietzky Universit\"at,
 26111 Oldenburg,
 Germany}
 \email{defant@mathematik.uni-oldenburg.de}

 \author[Galicer]{D.~Galicer}
 \address{Universidad Torcuato Di Tella. Departamento de Matem\'aticas y Estad\'istica.  IMAS-CONICET.
Av. Figueroa Alcorta 7350 (1428), Buenos Aires, Argentina.}
 \email{dgalicer$@$utdt.edu}

 \author[Mansilla]{M.~Mansilla}
 \address{Departamento de Matem\'{a}tica,
 Facultad de Cs. Exactas y Naturales, Universidad de Buenos Aires and IAM-CONICET. Saavedra 15 (C1083ACA) C.A.B.A., Argentina}
 \email{mmansilla$@$dm.uba.ar}

 \author[Masty{\l}o]{M.~Masty{\l}o}
 \address{Faculty of Mathematics and Computer Science, Adam Mickiewicz University, Pozna{\'n}, Uniwersytetu \linebreak
  Pozna{\'n}skiego 4,
 61-614 Pozna{\'n}, Poland}
 \email{mieczyslaw.mastylo$@$amu.edu.pl}

 \author[Muro]{S.~Muro}
\address{FCEIA, Universidad Nacional de Rosario and CIFASIS, CONICET, Ocampo $\&$ Esmeralda, S2000 Rosario, Argentina}
 \email{muro$@$cifasis-conicet.gov.ar}

\date{}

\begin{abstract}
We investigate local invariants and geometric phenomena for polynomial spaces of low degree on the $q$-ary Hamming scheme $C_q^N$, where $C_q$ denotes the cyclic group of order $q$. Our main analytic tool is a support-sensitive Bohnenblust--Hille inequality for spherical polynomial spaces, showing that the relevant complexity parameter is the support size of the monomials rather than their total degree. Equivalently, in the corresponding toroidal formulation, this leads to estimates for polynomials whose coordinate degrees are bounded by $q-1$, while the growth of the constants is governed by the interaction order of the variables. These inequalities yield applications to the learning theory of spherical low-level functions and also provide the basis for dimension-free comparisons between several classical local invariants, including Sidon constants, unconditional basis constants, and Gordon--Lewis constants. As a consequence, we obtain sharp asymptotic estimates for these invariants in the spherical setting, with analogous comparison and asymptotic results for homogeneous and tetrahedral polynomial spaces. We also study projection constants and the associated reproducing kernels. In the spherical case, suitably normalized Krawtchouk polynomials converge to Hermite polynomials under central-limit scaling, leading to explicit Gaussian limits and sharp asymptotic formulas. By contrast, in the homogeneous and tetrahedral settings a dichotomy appears between the Boolean case and the regime $q\ge3$, where the limiting behaviour is governed by moments of a circular complex Gaussian.
\end{abstract}



\maketitle
\pagestyle{plain}

\setcounter{tocdepth}{1}
\tableofcontents

\section{Introduction, preliminaries and main results}
Spaces of low-degree functions on product groups play a central role in Fourier analysis, probability, combinatorics, and local Banach space theory. In the setting of the $q$-ary Hamming scheme $C_q^N$ (where $C_q$ denotes the cyclic group of order $q$), such spaces provide a natural framework in which harmonic-analytic, probabilistic, and geometric phenomena interact in a highly structured way. Their study is closely connected with hypercontractivity, Boolean and cyclic-group Fourier analysis, discretization inequalities, extension and projection problems, and, more recently, questions arising in learning theory and theoretical computer science; see, for instance,
\cite{BannaiIto,DefantMastyloPerez,eskenazis2022learning,MacWilliamsSloane,ODonnell,SloteVolbergZhangBH,SloteVolbergZhangRemez}
and the references therein.

In this context, “low degree” refers broadly to functions whose Fourier
spectrum is supported on families of characters of bounded complexity,
such as characters of bounded degree or bounded support size, with
the relevant parameter fixed independently of the ambient dimension.

A central role in the present work is played by Sidon-type phenomena, which relate the summability of Fourier coefficients to supremum norms, and their interaction with other local invariants, such as unconditional basis constants and Gordon--Lewis constants. Although these quantities originate in rather different parts of harmonic analysis and Banach space theory, in the low-complexity regime, they become tightly connected through dimension-free comparison principles.

We investigate these phenomena for several low-degree polynomial spaces on the $q$-ary Hamming scheme $C_q^N$. More precisely,
we consider three natural classes: the spherical spaces
$\mathcal B_d(C_q^N)$, consisting of functions supported on characters
involving exactly $d$ coordinates; the homogeneous spaces
$\mathcal P_d(C_q^N)$, corresponding to characters of total degree
$d$; and the tetrahedral spaces $
\mathcal T_d(C_q^N)
=
\mathcal B_d(C_q^N)\cap \mathcal P_d(C_q^N)$,
which isolate the square-free homogeneous structure. All these spaces
are considered with the supremum norm inherited from $C(C_q^N)$.

A central analytic tool is a novel
Bohnenblust--Hille type inequality for spherical polynomial spaces, in
which the relevant parameter is the support size of the monomials
rather than their total degree. This is complemented by a support-sensitive
form of the dimension-free Remez transfer from the torus to the Hamming
scheme, obtained by keeping track of the active coordinates in the
randomization argument. This support-sensitive phenomenon leads to
asymptotically sharp estimates for Sidon constants and related local
invariants. Combined with dimension-free comparison principles, it yields
sharp asymptotic estimates for unconditional basis constants and
Gordon--Lewis constants. Analogous results also hold in the
homogeneous and tetrahedral settings.

From the toroidal point of view, this amounts to studying polynomial systems whose coordinate degrees are bounded by $q-1$. While such polynomials may have total degree as large as $(q-1)d$, our estimates are governed by the support size $d$ rather than the total degree. More precisely, the exponential growth is controlled by the interaction order of the variables: despite the potentially much larger total degree, the corresponding Bohnenblust--Hille constant grows only as $(C\sqrt q)^d$. Thus the interaction order rather than the total degree emerges as the governing complexity parameter.

This new inequality also leads to applications in learning
theory, particularly for spherical low-level functions on the Hamming scheme.
Roughly speaking, the problem consists in reconstructing an unknown function
from random evaluations using as few samples as possible, while guaranteeing a
prescribed approximation error with high probability.

Beyond these local and harmonic-analytic aspects, the spaces considered
here also display a rich asymptotic geometric structure. In particular,
we determine the exact asymptotic behaviour of the corresponding
projection constants, viewed as another fundamental family of local
invariants.

When $q=2$, the space $C_q^N$ is the Boolean cube $\{-1,1\}^N$,
and in this setting the spherical, homogeneous, and tetrahedral
structures coincide. For $q\ge3$, however, these spaces exhibit
substantially different asymptotic behaviours: while the spherical
asymptotics remain governed by Hermite polynomials, the homogeneous and
tetrahedral regimes are driven by genuinely complex oscillatory
phenomena.

Before explicitly stating our results, we review the basic definitions and notation that will be used throughout.

\subsection{Fourier analysis on the Hamming scheme}

We briefly recall the Fourier-analytic framework underlying the
spaces considered in this article. 

Let $q\ge 2$ and let
\[
C_q:=\{\omega^0,\omega^1,\dots,\omega^{q-1}\}\subset\mathbb{C},
\qquad \omega=e^{2\pi i/q},
\]
denote the multiplicative cyclic group of order $q$. We write $C_q$, rather than the more common notation $\mathbb{Z}_q$, to emphasize this multiplicative realization. For $N\in\mathbb{N}$, we consider the $N$-fold product
\[
C_q^N = \underbrace{C_q\times\cdots\times C_q}_{N\ \text{times}},
\]
known as the $q$-ary Hamming scheme, endowed with coordinatewise multiplication. This finite Abelian group has cardinality $q^N$, and its Haar measure
is the normalized counting measure. Consequently, for every function $f:C_q^N\to\mathbb{C}$,
\[
\mathbb{E}[f]=\frac{1}{q^N}\sum_{x\in C_q^N} f(x),
\qquad
\langle f,g\rangle=\mathbb{E}[f\,\overline{g}],
\]
define, respectively, the expectation and the standard inner product on $L^2(C_q^N)$.

The dual group $\widehat{C_q^N}$ consists of all characters $\chi:C_q^N\to\mathbb{T}$, that is, group homomorphisms into the torus $\mathbb{T}=\{z\in\mathbb{C}:|z|=1\}$. 
Identifying $\widehat{C_q^N}$ with $\mathbb{Z}_q^N=\{0,\dots,q-1\}^N$, each frequency $\bm m=(m_1,\dots,m_N)$ corresponds to the character
\[
\chi_{\bm m}(x)=x_1^{m_1}\cdots x_N^{m_N},
\qquad x=(x_1,\dots,x_N)\in C_q^N.
\]
The family $\{\chi_{\bm m}:\bm m\in\mathbb{Z}_q^N\}$ forms an orthonormal basis of $L^2(C_q^N)$. Consequently, every function $f:C_q^N\to\mathbb{C}$ admits a unique Fourier expansion
\[
f(x)=\sum_{\bm m\in\mathbb{Z}_q^N} \widehat{f}(\bm m)\,\chi_{\bm m}(x),
\qquad
\widehat{f}(\bm m)=\mathbb{E}\bigl[f\,\overline{\chi_{\bm m}}\bigr].
\]

\subsection*{Level--\texorpdfstring{$d$}{d} spherical polynomials}
For $\bm m=(m_1,\dots,m_N)\in\mathbb{Z}_q^N$, we define its support 
\[
\mathrm{supp}(\bm m)=\{j:\, m_j\neq 0\}.
\]
Thus, $\mathrm{supp}(\bm m)$ records the set of coordinates on which the corresponding character $\chi_{\bm m}$ depends nontrivially, whereas $|\mathrm{supp}(\bm m)|$ denotes the number of such coordinates.

The level--$d$ spherical space is defined by
\[
\mathcal B_d(C_q^N)=\operatorname{span}\{\chi_{\bm m}:|\mathrm{supp}(\bm m)|=d\}.
\]
Thus $\mathcal B_d(C_q^N)$ consists of the span of those characters depending nontrivially on exactly $d$
coordinates.
A function in this space is called a level-$d$ spherical polynomial. 

The terminology ``spherical'' is understood here in the sense of harmonic analysis on the
$q$-ary Hamming scheme. Indeed, as we shall see in Section~\ref{sec2}, the projection constant of this space is the $L_1$-norm of a function which depends only on the Hamming weight
\[
w(x):=\#\{k:x_k\neq 1\},
\]
so these spaces correspond to the radial layers associated with the distance structure of the
scheme. 

In this sense, the terminology is closer to that of zonal spherical functions and
Krawtchouk theory than to the classical theory of homogeneous polynomials on the Euclidean
sphere. Nevertheless, the classical Euclidean picture remains a useful guide, particularly when thinking of decompositions by degree.

\subsection*{Degree--\texorpdfstring{$d$}{d} homogeneous polynomials}

For $m\in \mathbb Z_q^N$ we define its degree by
\[
|m|:=\sum_{j=1}^N m_j,
\]
where the sum is interpreted in $\mathbb N_0$. The degree--$d$ homogeneous polynomial space is
\[
\mathcal P_d(C_q^N)=\operatorname{span}\{\chi_m:|m|=d\}.
\]
This is the natural analogue of the space of homogeneous polynomials of degree $d$ on
$\mathbb C^N$: the characters
\[
\chi_m(x)=x_1^{m_1}\cdots x_N^{m_N}
\]
play the role of monomials, and the condition $|m|=d$ selects those monomials of degree
$d$.

The space $\mathcal P_d(C_q^N)$ decomposes algebraically as
\begin{equation}\label{eq:homogeneous-support-decomposition}
\mathcal P_d(C_q^N)
=
\bigoplus_{k=0}^d \,\,\bigl(\mathcal B_k(C_q^N)\cap \mathcal  P_d(C_q^N)\bigr),
\end{equation}
where the summands correspond to monomials of degree $d$ supported on exactly $k$
coordinates. This decomposition reflects the fact that, for fixed $d$, different support sizes
encode different levels of interaction between coordinates.

Finally, one may also group the Fourier basis by degree and obtain the orthogonal decomposition
\[
L^2(C_q^N)=\bigoplus_{d=0}^{(q-1)N} \mathcal  P_d(C_q^N).
\]

\subsection*{Degree--$d$ tetrahedral polynomials}

The intersection of the spherical and homogeneous structures yields the degree--$d$ tetrahedral space
\[
\mathcal{T}_d(C_q^N)
   := \mathcal{B}_d(C_q^N)\cap \mathcal{P}_d(C_q^N),
\]
consisting of those characters whose degree is $d$ and whose support also has size $d$.  
Equivalently,
\[
\mathcal{T}_d(C_q^N)
   = \mathrm{span}\big\{\chi_{\bm m} : |\bm m|=d,\; |\mathrm{supp}(\bm m)|=d\big\},
\]
that is, monomials of degree $d$ in which the $d$ units of degree are distributed over $d$ distinct coordinates.  
These spaces isolate the square-free part of the homogeneous
polynomials and naturally arise in harmonic analysis on product
groups, discrete Fourier analysis, and polynomial approximation.

\subsection*{Spaces of level/degree at most $d$}

In many situations, it is natural to consider spaces whose spectrum is
supported on levels or degrees at most $d$. More precisely, for each $N$ we
define
\[
\mathcal B_{\le d}(C_q^N)
:=
\bigoplus_{k=0}^d \mathcal B_k(C_q^N),
\quad
\mathcal P_{\le d}(C_q^N)
:=
\bigoplus_{k=0}^d \mathcal P_k(C_q^N),
\quad
\mathcal T_{\le d}(C_q^N)
:=
\bigoplus_{k=0}^d \mathcal T_k(C_q^N).
\]
These correspond, respectively, to the spherical, homogeneous, and
tetrahedral low-complexity structures associated with the Hamming scheme.

\subsection*{Sidon constants}

Sidon constants play a central role in Fourier analysis on compact Abelian
groups. Let $G$ be a compact Abelian group and let $\widehat G$ denote its dual
group. Given $\Lambda\subset\widehat G$, the set $\Lambda$ is called a Sidon
set if there exists a constant $c>0$ such that, for every
$f\in\mathcal P_\Lambda(G)$,
\[
\sum_{\gamma\in\Lambda}|\widehat f(\gamma)|
\le
C\|f\|_{L_\infty(G)}.
\]
The least such constant $C$ is called the Sidon constant of $\Lambda$ and is
denoted by $\sid(\Lambda)$. Equivalently, we shall also write
\[
\sid\big(\mathcal P_\Lambda(G)\big)=\sid(\Lambda).
\]

Thus, for $\Lambda\subset\mathbb Z_q^N$, the Sidon constant of the system
$(\chi_{\bm m})_{\bm m\in\Lambda}$ is the least constant $c>0$ such that,
for every $f\in\mathcal P_{\Lambda}(C_q^N)$,
\[
\sum_{\bm m\in\Lambda}|\widehat f(\bm m)|
\le
C\|f\|_{L_\infty(C_q^N)}.
\]
This invariant links the $\ell_1$-mass of the Fourier coefficients over
$\Lambda\subset\mathbb Z_q^N$ with the supremum norm of the corresponding function. In this
way, it measures how the spectral distribution of $f$ influences its size.
This perspective is particularly well suited to low-complexity Fourier
structures and sparse spectral phenomena.

We will be particularly interested in the Sidon constants corresponding to the
spherical, homogeneous, and tetrahedral systems of characters of level or
degree $d$. These constants are closely connected with classical invariants
from local Banach space theory, such as the Gordon--Lewis constant and the
unconditional basis constant. Precise dimension free comparison results, as
well as the relevant definitions, are presented in Section~\ref{sec: local}.

\subsection*{Projection constants}

Projection constants are fundamental invariants in Banach space
theory, closely related to extension problems and approximation
phenomena; for background we refer to the monograph of
Tomczak--Jaegermann~\cite{tomczak1989banach}.

If $X$ is a subspace of a Banach space $Y$, the \emph{relative
projection constant} of $X$ in $Y$ is defined by
\[
\boldsymbol{\lambda}(X,Y)
=
\inf\big\{\|P\| \colon P \in \mathcal L(Y,X),\quad P|_X=\id_X\big\},
\]
where $\id_X$ denotes the identity operator on $X$. The \emph{absolute
projection constant} of $X$ is given by
\begin{equation}\label{definition}
\boldsymbol{\lambda}(X)
:=
\sup \boldsymbol{\lambda}(I(X),Z),
\end{equation}
where the supremum is taken over all Banach spaces $Z$ and all
isometric embeddings $I \colon X \to Z$.

\subsection*{Main results}
We now discuss the main contributions of the paper in greater detail.

Our first results concern local invariants associated with the spherical,
homogeneous, and tetrahedral spaces
$\mathcal B_d(C_q^N)$,
$\mathcal P_d(C_q^N)$,
and
$\mathcal T_d(C_q^N)$.
More precisely, we study their Sidon constants, unconditional basis
constants, and Gordon--Lewis constants.

A first ingredient for our proofs is a family of Remez-type transfer principles developed in a series of recent works by Becker, Klein, Slote, Volberg and Zhang  \cite{SloteVolbergZhangRemez,SloteVolbergZhangBH,becker2025dimension,KSVZ}. These results provide a mechanism for transferring norm estimates from finite cyclic grids to the full torus and have proved particularly useful in the study of Bohnenblust--Hille type inequalities on cyclic groups. In the present context, they also play a central role in obtaining learning-theoretic consequences. More precisely, their results show that, for low-degree
polynomials, the discrete set $C_q^N$ behaves as a ``norming set'' for the
torus $\TT^N$, up to constants depending only on $d$ and $q$. Consequently,
estimates for spaces of analytic polynomials over $\TT^N$ can be transferred to the Hamming scheme with constants independent of the ambient dimension.

We combine this with a recent principle from \cite{defant2026local}, where
Sidon constants, unconditional basis constants, and Gordon--Lewis constants
are compared for arbitrary families of monomials of uniformly bounded degree
on the torus. Together, these results yield dimension-free equivalences
between the preceding invariants in the spherical, homogeneous, and
tetrahedral settings; see Corollary~\ref{gl-versus-unc}.

Obtaining the asymptotically correct behaviour of these invariants requires
substantially more refined arguments. In the homogeneous and tetrahedral
settings, the asymptotically correct Sidon estimates follow from previously
known Bohnenblust--Hille inequalities for these families \cite{SloteVolbergZhangBH}.

The main analytic novelty is a  Bohnenblust--Hille inequality for
spherical polynomials, providing control of the Fourier coefficients in terms
of the supremum norm with constants depending exponentially on the support
level. The argument combines probabilistic and decoupling techniques with a
random block decomposition of the variables.

Starting from the toroidal model, where the monomials involve dependent powers
$z_i,z_i^2,\ldots,z_i^{q-1}$ of the same variable, we introduce a
dimension-free decoupling procedure replacing these powers by independent
variables. The coordinates are then randomly partitioned into $d$ labelled
blocks, and one isolates the monomials whose supports intersect every block.
This transforms the spherical structure into a tetrahedral multilinear one,
allowing the use of hypercontractive Bohnenblust--Hille inequalities.
Averaging over all labelled partitions recovers a fixed positive proportion of
the coefficients while avoiding the factorial losses associated with full
symmetrization.

The resulting estimate takes the following form. Theorem~\ref{BH cyclic low}
asserts that every function
\[
f
=
\sum_{\substack{
\bm m\in\mathbb Z_q^N\\
|\operatorname{supp}(\bm m)|\le d
}}
\widehat f(\bm m)\chi_{\bm m}
\in
\mathcal B_{\le d}(C_q^N)
\]
satisfies
\[
\Bigg(
\sum_{\substack{
\bm m\in\mathbb Z_q^N\\
|\operatorname{supp}(\bm m)|\le d
}}
|\widehat f(\bm m)|^{\frac{2d}{d+1}}
\Bigg)^{\frac{d+1}{2d}}
\le
D_1(q)^d
\|f\|_{L_\infty(C_q^N)}.
\]
Moreover, the  exponent $\frac{2d}{d+1}$ is optimal.

This support-sensitive inequality extends the known
Bohnenblust--Hille inequalities for bounded total degree \cite{becker2025dimension,SloteVolbergZhangBH}. Indeed, one has
\begin{equation}\label{eq: important inclusion}
\mathcal P_{\le d}(C_q^N)\subset \mathcal B_{\le d}(C_q^N),
\end{equation}
since every monomial of total degree at most $d$ necessarily depends on at
most $d$ coordinates.

Combined with the dimension-free equivalences between local invariants
described earlier, all these estimates lead to the asymptotically correct
behaviour of the preceding quantities. More precisely, if
$\eta$ denotes either the Sidon constant, the unconditional basis constant,
or the Gordon--Lewis constant, then
Theorem~\ref{thm:invariants} shows that
\[
\eta\big(\mathcal B_d(C_q^N)\big),
\quad
\eta\big(\mathcal P_d(C_q^N)\big),
\quad
\eta\big(\mathcal T_d(C_q^N)\big)
\sim_{c(q)^d}
\Big(\frac{N}{d}\Big)^{\frac{d-1}{2}}.
\]

As a byproduct of the support-sensitive Bohnenblust--Hille inequality above, we also obtain
learning-theoretic consequences for low-level functions on Hamming schemes. The Boolean case $q=2$ corresponds to the classical low-degree setting studied
by Eskenazis and Ivanisvili~\cite{eskenazis2022learning}, while learning
results for bounded total degree on products of cyclic groups were later
developed in~\cite{becker2025dimension}. In our setting, functions of level at
most $d$ may still have total degree as large as $(q-1)d$, so the previous
approach applies only with this larger parameter.  Consequently,
although both approaches yield logarithmic dependence on the ambient dimension
$N$, the estimates obtained here improve the dependence on the accuracy
parameter $\varepsilon$, replacing factors of the form $
\varepsilon^{-((q-1)d+1)}
$
by
$
\varepsilon^{-(d+1)}.
$
All these estimates are established in
Theorem~\ref{learning spherical}. As mentioned above, every polynomial of total degree at most $d$
automatically has level at most $d$, so our framework strictly extends the
one considered in~\cite{becker2025dimension}.

Projection constants constitute another fundamental class of local
invariants in the study of the spherical, homogeneous, and tetrahedral
spaces introduced above, viewed as subspaces of $C(C_q^N)$.

In view of the classical Kadets--Snobar theorem~\cite{kadecsnobar}, which
asserts that the projection constant of every finite-dimensional Banach space
$X$ is bounded above by $\sqrt{\dim X}$, it is natural to ask whether this
estimate is asymptotically sharp in the present setting. We show that this is
indeed the case and determine the limits, as $N\to\infty$, of the normalized
quantities
\[
\frac{\boldsymbol{\lambda}(X_N)}{\sqrt{\dim X_N}},
\]
when $X_N$ ranges over the spaces $\mathcal B_d(C_q^N)$,
$\mathcal P_d(C_q^N)$, and $\mathcal T_d(C_q^N)$.

The spherical setting exhibits a universal behaviour. Here,
$Z\sim\mathcal N(0,1)$ denotes a standard Gaussian random variable and
$\mathrm{He}_d$ is the $d$-th probabilists' Hermite polynomial
(see \cite[Ch.~5]{Szego1975} and \cite[\S18.23]{DLMF} for background).
Theorem~\ref{thm: main spherical} shows that
\[
\lim_{N\to\infty}
\frac{\boldsymbol{\lambda}\big(\mathcal B_d(C_q^N)\big)}
{\sqrt{\dim \mathcal B_d(C_q^N)}}
=
\frac{\mathbb E\big|\mathrm{He}_d(Z)\big|}{\sqrt{d!}},
\]
where
\[
\dim \mathcal B_d(C_q^N)=\binom{N}{d}(q-1)^d.
\]

Thus the asymptotic behaviour is universal across all cyclic groups $C_q$.
The Gaussian constant above also admits a precise asymptotic expansion as
$d\to\infty$, namely
\[
\frac{1}{\sqrt{d!}}\,\mathbb E\big|\mathrm{He}_d(Z)\big|
=
\frac{2^{7/4}}{\pi^{5/4}}\,\frac{1}{d^{1/4}}
\bigg(1+O\bigg(\frac{1}{d^2}\bigg)\bigg).
\]

This follows from the asymptotic formula for the $L_1$-norm of Hermite
polynomials due to Larsson--Cohn~\cite[Remarks~2.6 and~3.2]{larsson2002p}.

The underlying mechanism behind this phenomenon is that the projection
constant can be represented as the $L_1$-norm of a function depending only on
the Hamming weight $w(x)=\#\{j:\,x_j\neq 1\}$, and therefore admits an
explicit description in terms of Krawtchouk polynomials, the orthogonal
polynomials naturally associated with the Hamming scheme and the underlying
binomial distribution. After suitable normalization, these polynomials
converge to Hermite polynomials, which are orthogonal with respect to the
Gaussian measure. This asymptotic transition reflects a central limit
phenomenon for the Hamming weight and heuristically explains the Gaussian
limit above. 

We next consider the homogeneous and tetrahedral spaces. In contrast with
the spherical setting, a genuine dichotomy emerges. When $q=2$, the
spherical, homogeneous, and tetrahedral structures coincide, so the
corresponding asymptotic behaviour follows from the spherical case above; see
also \cite{defant2024asymptotic} for a different
approach. For $q\ge3$, however, the situation changes substantially. More precisely, Theorem~\ref{thm:asymptotic-projection-Pd} shows that
\[
\lim_{N\to\infty}
\frac{\boldsymbol{\lambda}\big(\mathcal P_d(C_q^N)\big)}
     {\sqrt{\dim \mathcal P_d(C_q^N)}}
=
\lim_{N\to\infty}
\frac{\boldsymbol{\lambda}\big(\mathcal T_d(C_q^N)\big)}
     {\sqrt{\dim \mathcal T_d(C_q^N)}} =
\frac{\Gamma\big(1+\frac{d}{2}\big)}{\sqrt{d!}}.
\]

The proof combines asymptotic analysis of character sums with
probabilistic limit arguments. In contrast with the spherical setting, the
asymptotics are no longer driven by ``radial'' structures, but by the
interaction and cancellation of oscillatory complex phases arising from the
Fourier characters. More precisely, the relevant quantities involve large
character sums whose terms take values among the $q$-th roots of unity.
These highly oscillatory sums exhibit substantial combinatorial cancellation
and, after normalization, converge asymptotically to circular complex
Gaussian limits. The latter can be computed explicitly, leading to the
constant which appears above.

A final phenomenon concerns spaces whose spectrum is supported on levels or
degrees at most $d$. In this case, a common structural principle governs all
three settings considered above: after normalization by the square root of the
dimension, the asymptotic behaviour is entirely determined by the top layer,
namely the component corresponding to level or degree exactly $d$. This
top-layer principle reappears throughout the paper, governing the asymptotic
behaviour of all the local invariants studied here.

\section{Support-sensitive Bohnenblust--Hille inequalities and learning applications}

The main goal of this section is to establish support-sensitive
Bohnenblust--Hille inequalities for spherical low-level functions on Hamming
schemes and to derive the corresponding learning-theoretic consequences.

A recent breakthrough of Becker, Klein, Slote, Volberg, and Zhang
\cite{becker2025dimension}; see also
\cite{KSVZ},
 \cite{SloteVolbergZhangBH}, established Bohnenblust--Hille inequalities for
functions on products of cyclic groups. Their approach combines a remarkable Remez-type theorem for the Hamming
scheme with the classical Bohnenblust--Hille inequality on the torus. More
precisely, starting from the known estimate in the toroidal setting, one
transfers it to the discrete model through a dimension-free comparison
between supremum norms on $\TT^N$ and on $C_q^N$.

In the spherical setting considered here, however, the relevant parameter is
the support size of the monomials rather than their total degree, and the
preceding approach is no longer sufficient. This leads to a substantially
more delicate situation requiring additional ideas. Nevertheless, the torus
still plays an important role in the argument, and the corresponding
Remez-type transfer principle will therefore be recalled first.

\subsection{Remez-type transfer from the torus to the Hamming scheme} \label{Remez}

Let $\TT=\{z\in\C\colon |z|=1\}$ be endowed with its normalized Haar
probability measure, and let $\TT^N$ be equipped with the product Haar measure.
For integers $q\ge2$ and $d\ge1$, we consider two natural spaces of analytic
polynomials on $\TT^N$ associated with the $q$-ary setting.

First, the $q$-capped homogeneous space of degree $d$ is defined by
\[
\mathcal P_{d,q}(\TT^N)
:=
\operatorname{span}\big\{
z^\alpha\colon
\alpha\in\{0,1,\dots,q-1\}^N,\ |\alpha|=d
\big\}
\subset C(\TT^N).
\]
In parallel, we introduce the toric analogue of the spherical space,
\begin{equation}\label{toric analogue}
\mathcal B_{d,q}(\TT^N)
:=
\operatorname{span}\big\{
z^\alpha\colon \alpha\in\{0,1,\dots,q-1\}^N,\
|\operatorname{supp}(\alpha)|=d
\big\}
\subset C(\TT^N).
\end{equation}
The spherical space is nontrivial only when $1\le d\le N$; in what follows,
whenever spherical spaces of level $d$ are involved, this is the relevant
range.

Note that every polynomial in $\mathcal P_{d,q}(\TT^N)$ or
$\mathcal B_{d,q}(\TT^N)$ restricts to a function on $C_q^N$. Conversely,
every function on $C_q^N$ has a unique Fourier expansion with frequencies
in $\{0,\dots,q-1\}^N$, and this expansion gives a unique extension to an
analytic polynomial on $\TT^N$ whose degree in each variable is at most
$q-1$. Thus $\mathcal P_{d,q}(\TT^N)$ and $\mathcal B_{d,q}(\TT^N)$ provide
toric realizations of $\mathcal P_d(C_q^N)$ and $\mathcal B_d(C_q^N)$,
respectively.

In addition to the exact-degree spaces introduced above, we shall also work
with their cumulative counterparts, denoted by
\[
\mathcal P_{\le d,q}(\TT^N)
\quad\text{and}\quad
\mathcal B_{\le d,q}(\TT^N),
\]
corresponding to degree and support level at most $d$, respectively. Their definitions are entirely analogous.

The main result of \cite{becker2025dimension,KSVZ,SloteVolbergZhangRemez} shows that the supremum norm on the whole torus is controlled by the supremum norm on the discrete grid $C_q^N$, with a constant independent of the ambient dimension. In its total-degree form, this result asserts that there exists a universal constant $C>0$ such that, for every $q\ge2$, $d\ge1$, $N\ge1$, and every polynomial $P\in \mathcal P_{\le d,q}(\TT^N)$, that is, 
\[
P(z)=\sum_{\substack{\alpha\in\{0,\ldots,q-1\}^N\\ |\alpha|\le d}}
a_\alpha z^\alpha, \quad z \in \TT^N,
\]
one has
\begin{equation}\label{thm:SVZ-remez}
\|P\|_{L_\infty(\TT^N)} \le (C\log q)^d \|P\|_{L_\infty(C_q^N)}.
\end{equation}
Equivalently, the grid $C_q^N$ is a norming set for $q$-capped analytic polynomials of total degree at most $d$, with a constant independent of $N$.

Applying this estimate as a black box to the spherical low-level space gives a weaker bound. Indeed, every monomial $z^\alpha$ appearing in $\mathcal B_{\le d,q}(\TT^N)$ satisfies
\[
|\alpha|
\le
(q-1)|\operatorname{supp}(\alpha)|
\le
(q-1)d.
\]
Hence every polynomial in $\mathcal B_{\le d,q}(\TT^N)$ has total degree at most $(q-1)d$, and \eqref{thm:SVZ-remez} gives
\begin{equation}\label{eq:Remez-spherical}
\|P\|_{L_\infty(\TT^N)}
\le
(C\log q)^{(q-1)d}
\|P\|_{L_\infty(C_q^N)},
\quad
P\in\mathcal B_{\le d,q}(\TT^N).
\end{equation}

\noindent
A support-sensitive version of this removes this loss. Although this formulation does not appear explicitly in \cite{KSVZ}, it is implicit in their proof of the dimension-free Remez inequality. Indeed, the relevant part of the randomization argument depends on the number of active coordinates of the monomials involved, rather than on their total degree. Keeping track of this parameter throughout the proof yields a support-sensitive refinement. Since this formulation will be needed later for spherical levels, we present the details here, following closely the notation and randomization scheme of \cite{KSVZ}.

Fix $z=(z_1,\ldots,z_N)\in\TT^N$. The one-dimensional interpolation lemma from \cite[Lemma 7]{KSVZ} gives a universal constant $B>0$ with the following property. For each coordinate $z_j$, there are step functions 
\[ 
r\colon [0,D]\to\{1,i,-1,-i\}, \quad w_j\colon [0,D]\to C_q, 
\] 
with $D\le C_0\log q$ for a universal constant $C_0>0$, such that, if $T$ is uniformly distributed on $[0,D]$, then 
\[ 
z_j^k = D\,\mathbb E\big[r(T)w_j(T)^k\big], \quad k=0,1,\ldots,q-1. 
\]
The phase function $r$ is common to all coordinates, whereas $w_j$ depends on $z_j$.

For an integer $m\ge1$, let $T_1,\ldots,T_m$ be independent random variables uniformly distributed on $[0,D]$, and let
\[
\sigma\colon\{1,\ldots,N\}\to\{1,\ldots,m\}
\]
be a uniformly random map, independent of $T_1,\ldots,T_m$. Define
\[
R_m:=\prod_{\ell=1}^m r(T_\ell)
\]
and
\[
W_m=(W_{m,1},\ldots,W_{m,N})\in C_q^N,
\quad
W_{m,j}:=w_j(T_{\sigma(j)}).
\]
Then $|R_m|=1$ and $W_m\in C_q^N$.

The next lemma isolates the monomial representation obtained from the randomization construction.

\begin{lemma}\label{lem:support-sensitive-monomial-representation} Let $q\ge2$, let $z\in\TT^N$, let $\alpha\in\{0,\ldots,q-1\}^N$, and let $S=\operatorname{supp}(\alpha)$ with $s:=|S|$. Then there exists a polynomial $p_\alpha$ such that 
\[ 
p_\alpha(0)=0, \quad \deg p_\alpha<s,
\] 
with the convention that the zero polynomial has degree $-\infty$, and, for every integer $m\ge1$,
\[
z^\alpha = D^m\mathbb E\big[R_mW_m^\alpha\big] + p_\alpha(1/m). 
\] 
\end{lemma}

\begin{proof}
The case $s=0$ corresponds to $\alpha=0$ and is immediate, since
\[
D^m\mathbb E[R_m]=1.
\]
Thus we may take $p_\alpha=0$. We therefore assume that $s\ge1$.

Only the coordinates in $S=\operatorname{supp}(\alpha)$ are relevant. The random map $\sigma$ induces a partition of $S$: two points of $S$ belong to the same block precisely when they are sent by $\sigma$ to the same element of $\{1,\ldots,m\}$.

Let $\pi$ be a partition of $S$ with $t$ blocks. On the event that $\sigma$ induces $\pi$, the expectation has the form 
\[ 
\mathbb E\big[R_mW_m^\alpha\mid \sigma \text{ induces } \pi\big] 
= D^{-m}E_\pi(z,\alpha), 
\] 
where $E_\pi(z,\alpha)$ is independent of $m$. In the special case of the singleton partition $\pi_0$, that is, when $\sigma$ is injective on $S$, the one-dimensional identities above multiply without collision and give \[ 
E_{\pi_0}(z,\alpha)=z^\alpha. 
\]

We now examine the probability that $\sigma$ induces a fixed partition $\pi$ with $t$ blocks. This probability is
\[
\mathbb P(\sigma \text{ induces } \pi)
=
\frac{m(m-1)\cdots(m-t+1)}{m^s}.
\]
Equivalently, putting $x=1/m$, we can write
\[
\frac{m(m-1)\cdots(m-t+1)}{m^s}
=
x^{s-t}\prod_{r=0}^{t-1}(1-rx).
\]
If $t=s$, corresponding to the singleton partition, this is of the form
\[
1+q_{\pi_0}(x),
\]
where
\[
q_{\pi_0}(0)=0,
\quad
\deg q_{\pi_0}<s.
\]
If $t<s$, then
\[
x^{s-t}\prod_{r=0}^{t-1}(1-rx)
\]
has zero constant term and degree at most $s-1$.

Using the law of total probability over all partitions of $S$, we obtain
\[
D^m\mathbb E\big[R_mW_m^\alpha\big]
=
z^\alpha+h_\alpha(1/m),
\]
where
\[
h_\alpha(0)=0,
\quad
\deg h_\alpha<s.
\]
Setting $p_\alpha=-h_\alpha$ gives the desired identity.
\end{proof}

The preceding lemma contains the key support-sensitive ingredient. Combining it with a finite interpolation argument in the auxiliary parameter $1/m$ eliminates the collision terms and leads to the desired Remez estimate. Although the estimate below is not stated in this form in \cite{KSVZ}, it follows by retaining the support-size parameter in the proof of \cite[Theorem 5]{KSVZ}.

\begin{theorem}\label{thm:remez-support}
There exists a universal constant $C>0$ such that, for every $q\ge2$, $d\ge0$, $N\ge1$, and every polynomial $P\in\mathcal B_{\le d,q}(\TT^N)$, one has
\[
\|P\|_{L_\infty(\TT^N)}
\le
(C\log q)^d
\|P\|_{L_\infty(C_q^N)}.
\]
\end{theorem}

\begin{proof}
The case $d=0$ is immediate, since then $P$ is constant and hence
\[
\|P\|_{L_\infty(\TT^N)}
=
\|P\|_{L_\infty(C_q^N)}.
\]
We therefore assume that $d\ge1$.

Fix $z\in\TT^N$. Applying Lemma~\ref{lem:support-sensitive-monomial-representation} to each monomial appearing in $P$ and summing, we obtain, for every $m=1,\ldots,d$,
\[
P(z)
=
D^m\mathbb E\big[R_mP(W_m)\big]
+
p_z(1/m),
\]
where
\[
p_z(0)=0,
\quad
\deg p_z<d.
\]
Indeed, every monomial appearing in $P$ satisfies $|\operatorname{supp}(\alpha)|\le d$, and hence its error term has degree strictly smaller than $d$.

Choose scalars $\lambda_1,\ldots,\lambda_d$ such that
\[
\sum_{m=1}^d\lambda_m=1,
\quad
\sum_{m=1}^d\lambda_m m^{-r}=0,
\quad
r=1,\ldots,d-1.
\]
Equivalently, the functional
\[
p\mapsto\sum_{m=1}^d\lambda_m p(1/m)
\]
annihilates every polynomial $p$ with $p(0)=0$ and $\deg p<d$, while preserving constants. The explicit choice is
\[
\lambda_m
=
(-1)^{d-m}\frac{m^d}{m!(d-m)!},
\quad
m=1,\ldots,d.
\]
In particular, these coefficients satisfy
\[
\sum_{m=1}^d|\lambda_m|
\le
\exp(C_1d)
\]
with a universal constant $C_1>0$.

Multiplying the identity for $P(z)$ by $\lambda_m$ and summing over $m=1,\ldots,d$, the error term disappears and we get
\[
P(z)
=
\sum_{m=1}^d
\lambda_mD^m\mathbb E\big[R_mP(W_m)\big].
\]
Since $W_m\in C_q^N$ and $|R_m|=1$, it follows that
\[
|P(z)|
\le
\sum_{m=1}^d|\lambda_m|D^m\|P\|_{L_\infty(C_q^N)}
\le
D^d\left(\sum_{m=1}^d|\lambda_m|\right)
\|P\|_{L_\infty(C_q^N)}.
\]
Using $D\le C_0\log q$ and absorbing the factor $\exp(C_1d)$ into the base of the exponential, we obtain
\[
|P(z)|
\le
(C\log q)^d\|P\|_{L_\infty(C_q^N)}.
\]
Taking the supremum over $z\in\TT^N$ completes the proof.
\end{proof}

We conclude this subsection with a complementary discretization estimate in the fine-grid regime $q>d$. It shows that when the sampling grid is substantially finer than the degree of the polynomial, the constant in the Remez-type comparison can be chosen close to $1$. The proof is based on a multidimensional version of Stečkin's lemma due to de la Chevrotière~\cite{de2009finding}.

\begin{proposition}\label{prop: q>d}
Assume that $q>d$. Then every $P\in\mathcal P_{\le d,q}(\TT^N)$ satisfies
\[
\|P\|_{L_\infty(\TT^N)}
\le
\sec\!\left(\frac{\pi d}{2q}\right)
\|P\|_{L_\infty(C_q^N)}.
\]
\end{proposition}

\begin{proof}
Choose $t^0=(t_1^0,\ldots,t_N^0)\in\mathbb R^N$ so that
\[
\big|P\big(e^{it_1^0},\ldots,e^{it_N^0}\big)\big|
= \|P\|_{L_\infty(\TT^N)}.
\]
Since $C_q$ consists of the $q$-th roots of unity, for each coordinate $e^{it_j^0}$ we may choose $\omega_j\in C_q$ whose angular distance from $e^{it_j^0}$ is at most $\pi/q$. Hence there are real numbers $t_1,\ldots,t_N$ such that
\[
\omega_j=e^{i(t_j^0+t_j)} \quad\text{and}\quad
|t_j|\le\frac{\pi}{q}
\]
for each $j\in \{1,\ldots,N\}$. In particular, for the fixed vector $t=(t_1,\ldots,t_N)$ we have $\max_j|t_j|\le\pi/q$.

We now translate the maximum point to the origin and consider the real trigonometric polynomial
\[
Q(s):=2\frac{
\big|P\big(e^{i(t_1^0+s_1)},\ldots,e^{i(t_N^0+s_N)}\big)\big|^2
}{
\|P\|_{L_\infty(\TT^N)}^2
}
-1,
\quad
s=(s_1,\ldots,s_N)\in\mathbb R^N.
\]
By the argument in \cite[Section~4]{de2009finding}, the polynomial $Q$ has total degree at most $d$. Moreover, by construction, $Q(0)=1$, and the maximality of the chosen point implies that $\|Q\|_\infty=1$. Since $q>d$, the fixed vector $t$ satisfies $\max_j|t_j|\le\pi/q<\pi/d$. Therefore the multidimensional Stečkin lemma of de la Chevrotière~\cite[Theorem~3.2]{de2009finding} applies to $Q$ at this vector. Hence
\[
Q(t)
\ge
\cos\!\big(d\max_j|t_j|\big)
\ge
\cos\!\left(\frac{\pi d}{q}\right).
\]
Recalling the definition of $Q$ and the identity
\[
\omega
=
\big(e^{i(t_1^0+t_1)},\ldots,e^{i(t_N^0+t_N)}\big),
\]
we obtain
\[
2\frac{|P(\omega)|^2}{\|P\|_{L_\infty(\TT^N)}^2}-1
\ge
\cos\!\left(\frac{\pi d}{q}\right).
\]
Using the elementary identity $(1+\cos\theta)/2=\cos^2(\theta/2)$, this gives
\[
|P(\omega)|
\ge
\|P\|_{L_\infty(\TT^N)}
\cos\!\left(\frac{\pi d}{2q}\right).
\]
Finally, since $\omega\in C_q^N$, we have
\[
\|P\|_{L_\infty(C_q^N)}
\ge
|P(\omega)|
\ge
\|P\|_{L_\infty(\TT^N)}
\cos\!\left(\frac{\pi d}{2q}\right),
\]
which is equivalent to the desired estimate.
\end{proof}

\subsection{Support-sensitive Bohnenblust--Hille inequalities on exact levels}

A first step in the argument is to combine block decompositions with decoupling techniques in order to reduce the problem to multilinear tetrahedral forms. This leads to a Bohnenblust--Hille type estimate for the space
$\mathcal B_{d,q}(\TT^N)$.

\begin{theorem}\label{thm:BH-toric-spherical}
For every $q\ge2$ there exists a constant $c(q)=O(\sqrt q)$ such that, for every
$d,N\in\mathbb N$ and every $P\in\mathcal B_{d,q}(\TT^N)$ of the form
\[
P(z)
=
\sum_{\substack{\alpha\in\{0,\ldots,q-1\}^N\\
|\operatorname{supp}(\alpha)|=d}}
a_\alpha z^\alpha,
\quad z\in\TT^N,
\]
one has
\[
\bigg(
\sum_{\substack{\alpha\in\{0,\ldots,q-1\}^N\\
|\operatorname{supp}(\alpha)|=d}}
|a_\alpha|^{\frac{2d}{d+1}}
\bigg)^{\frac{d+1}{2d}}
\le
c(q)^d
\|P\|_{L_\infty(\TT^N)}.
\]
\end{theorem}

The important point is that the constant grows exponentially in the support
size $d$, with base depending only on $q$, and not on $N$. This is the feature
needed in applications to spherical polynomial spaces on the cyclic group
$C_q^N$.

The proof will be given in four steps. We first prove a local finite-dimensional estimate based on a decoupling
procedure. Roughly speaking, it shows that the finite alphabet
$z,z^2,\ldots,z^{q-1}$ can be replaced by independent variables
$y_1,\ldots,y_{q-1}$ at a cost depending only on $q$.

\begin{lemma}\label{lem:local-cloning-q}
For every $q\ge 2$ there exists a constant $k(q)=O(\sqrt q)$ such that,
for every finite set $A$, the linear map
\[
U_A:\operatorname{span}\{z_i^r: i\in A,\ 1\le r\le q-1\}
\to
\operatorname{span}\{y_{i,r}: i\in A,\ 1\le r\le q-1\},
\qquad U_A(z_i^r)=y_{i,r},
\]
is well defined and has norm at most $k(q)$, when the domain and range are endowed with the
supremum norms inherited from $C(\mathbb T^A)$ and
$C(\mathbb T^{A\times\{1,\ldots,q-1\}})$, respectively. 
\end{lemma}

% \begin{lemma}\label{lem:local-cloning-q}
% For every $q\ge 2$ there exists a constant $k(q)=O(\sqrt q)$ such that, for every
% finite set $A$, the operator
% \[
% U_A:
% \operatorname{span}\{z_i^r: i\in A,\ 1\le r\le q-1\}
% \longrightarrow
% \operatorname{span}\{y_{i,r}: i\in A,\ 1\le r\le q-1\},
% \qquad
% U_A(z_i^r)=y_{i,r},
% \]
% is well defined and satisfies $\|U_A\|\le k(q)$.
% \end{lemma}

\begin{proof}
Let
\[
f(z)=\sum_{i\in A}\sum_{r=1}^{q-1} c_{i,r} z_i^r, \qquad z \in \mathbb T^A\, .
\]
For each $i\in A$, write
\[
P_i(w)=\sum_{r=1}^{q-1} c_{i,r}w^r,
\qquad w\in\mathbb T.
\]
Then
\[
f(z)=\sum_{i\in A}P_i(z_i).
\]
We now compare $\sum_i \|P_i\|_{L^\infty(\mathbb T)}$ with $\|f\|_\infty$. For each
$\theta\in[0,2\pi]$, define
\[
h_i(\theta)
:=
\sup_{w\in\mathbb T}
\operatorname{Re}\big(e^{-i\theta}P_i(w)\big).
\]
Since \(P_i\) has no constant term, the function
\[
g_\theta(w):=\operatorname{Re}(e^{-i\theta}P_i(w)),
\qquad w\in\mathbb T.
\]
has mean zero on \(\mathbb T\). Hence
\[
h_i(\theta)=\sup_{w\in\mathbb T} g_\theta(w)\ge 0,
\qquad \theta \in [0, 2\pi].
\]

Choose $w_i^0\in\mathbb T$ such that
\[
|P_i(w_i^0)|=\|P_i\|_{L^\infty(\mathbb T)}.
\]
Write
\[
P_i(w_i^0)=\|P_i\|_{L^\infty(\mathbb T)} e^{i\alpha_i}.
\]
Then for all $\theta\in[0,2\pi]$
\[
h_i(\theta)
\ge
\operatorname{Re}\big(e^{-i\theta}P_i(w_i^0)\big)
= \|P_i\|_{L^\infty(\mathbb T)}
\cos(\alpha_i-\theta).
\]
Since $h_i(\theta)\ge 0$, it follows that
\[
h_i(\theta)
\ge
\|P_i\|_{L^\infty(\mathbb T)}\max\{\cos(\alpha_i-\theta),0\}.
\]
Averaging with respect to $\theta$, we obtain
\[
\frac{1}{2\pi}\int_0^{2\pi} h_i(\theta)\,d\theta
\ge
\|P_i\|_{L^\infty(\mathbb T)}
\frac{1}{2\pi}\int_0^{2\pi}\max\{\cos t,0\}\,dt
=
\frac{1}{\pi} \|P_i\|_{L^\infty(\mathbb T)}.
\]
Therefore
\[
\frac{1}{2\pi}\int_0^{2\pi}\sum_{i\in A}h_i(\theta)\,d\theta
\ge
\frac{1}{\pi}\sum_{i\in A}\|P_i\|_{L^\infty(\mathbb T)}.
\]
Hence there exists $\theta_0\in[0,2\pi]$ such that
\[
\sum_{i\in A}h_i(\theta_0)
\ge
\frac{1}{\pi}\sum_{i\in A}\|P_i\|_{L^\infty(\mathbb T)}.
\]
For this $\theta_0$, choose $z_i\in\mathbb T$ such that
\[
\operatorname{Re}\big(e^{-i\theta_0}P_i(z_i)\big)
=
h_i(\theta_0).
\]
Then
\[
\|f\|_{L^\infty(\mathbb T^A)}
\ge
|f(z)|
\ge
\operatorname{Re}\big(e^{-i\theta_0}f(z)\big)
=
\sum_{i\in A}h_i(\theta_0)
\ge
\frac{1}{\pi}\sum_{i\in A}\|P_i\|_{L^\infty(\mathbb T)}.
\]
Thus
\[
\sum_{i\in A}\|P_i\|_{L^\infty(\mathbb T)}
\le
\pi \|f\|_{L^\infty(\mathbb T^A)}.
\]
On the other hand, by Cauchy's inequality and Parseval's identity,
\[
\sum_{r=1}^{q-1}|c_{i,r}|
\le
\sqrt{q-1}
\left(\sum_{r=1}^{q-1}|c_{i,r}|^2\right)^{1/2}
=
\sqrt{q-1}\,\|P_i\|_{L^2(\mathbb T)}
\le
\sqrt{q-1}\,\|P_i\|_{L^\infty(\mathbb T)}.
\]
Combining both estimates, we get
\[
\sum_{i\in A}\sum_{r=1}^{q-1}|c_{i,r}|
\le
\sqrt{q-1}\sum_{i\in A}\|P_i\|_{L^\infty(\mathbb T)}
\le
\pi\sqrt{q-1}\,\|f\|_{L^\infty(\mathbb T^A)}.
\]
Finally,
\[
\|U_A f\|_{L^\infty(\mathbb T^{A\times\{1,\ldots,q-1\}})}
=
\sup_{|y_{i,r}|=1}
\left|
\sum_{i\in A}\sum_{r=1}^{q-1} c_{i,r}y_{i,r}
\right|
=
\sum_{i\in A}\sum_{r=1}^{q-1}|c_{i,r}|.
\]
Therefore
\[
\|U_A f\|_\infty
\le
\pi\sqrt{q-1}\,\|f\|_\infty.
\]
The proof is complete.
\end{proof}

\medskip

\begin{remark}
The estimate in Lemma~\ref{lem:local-cloning-q} is sharp in its dependence on \(q\), up to a
universal constant. Indeed, if \(A\neq\emptyset\), then, fixing
\(i_0\in A\) and restricting \(U_A\) to
\(\operatorname{span}\{z_{i_0}^r:1\le r\le q-1\}\), we get
\[
 \|U_A\|
 \ge
 \sup_{(c_r)\neq 0}
 \frac{\sum_{r=1}^{q-1}|c_r|}
 {\left\|\sum_{r=1}^{q-1} c_r z^r\right\|_{L_\infty(\mathbb T)}} .
\]
The quantity on the right is the Sidon constant of the system
\(\{z,z^2,\ldots,z^{q-1}\}\) on \(\mathbb T\), which is known to be bounded
from below by \(c\sqrt q\) with a universal constant \(c>0\). Hence
\[
 \|U_A\|\ge c\sqrt q
\]
for every non-empty finite set \(A\). Together with  Lemma~\ref{lem:local-cloning-q}, this shows
that \(\|U_A\|\asymp \sqrt q\), uniformly in \(A\neq\emptyset\).
\end{remark}

We next isolate the monomials whose support meets every block of a labelled
partition. For fixed $d,q,N$, we shall use the notation
\[
\Lambda_{d,q,N}
:=
\big\{
\alpha\in\{0,\ldots,q-1\}^N\colon
|\operatorname{supp}(\alpha)|=d
\big\}.
\]

\begin{lemma}\label{lem:block-full-projection-q}
Let $[N]=A_1\dot\cup\cdots\dot\cup A_d$ be a disjoint 
partition into $d$ labelled,
possibly empty, blocks. Let $P\in \mathcal B_{d,q}(\TT^N)$ be given by
\[
P(z)
=
\sum_{\alpha\in\Lambda_{d,q,N}} a_\alpha z^\alpha,
\quad z\in\TT^N.
\]
Let $P_A$ denote the part of $P$ consisting of those monomials whose support
meets every block $A_1,\ldots,A_d$. Then
\[
 \|P_A\|_{L_\infty(\TT^N)}
 \le
 2^d
 \|P\|_{L_\infty(\TT^N)}.
\]
Moreover, every monomial appearing in $P_A$ meets each block in exactly one
coordinate.
\end{lemma}

\begin{proof}
For each block $A_k$, let $E_{A_k}$ denote the conditional expectation obtained
by averaging over the variables in $A_k$. Equivalently, $E_{A_k}$ is the
Fourier projection onto the monomials which do not involve variables from
$A_k$. Hence $I-E_{A_k}$ is the projection onto the monomials which involve at
least one variable from $A_k$.

With the convention that $E_\varnothing=I$, we have
\[
 P_A
 =
 (I-E_{A_1})\cdots(I-E_{A_d})P.
\]
Indeed, the projections $E_{A_k}$ commute, and the product above keeps exactly
those monomials whose support meets every block of the partition. Since each
$E_{A_k}$ is contractive on $L_\infty(\TT^N)$,
\[
\|I-E_{A_k}\|_{L_\infty\to L_\infty}
\le
1+\|E_{A_k}\|
\le
2.
\]
Therefore
\[
\|P_A\|_{L_\infty(\TT^N)}
\le
2^d\|P\|_{L_\infty(\TT^N)}.
\]

Finally, every monomial in $P$ has support of cardinality exactly $d$. If such
a monomial meets all $d$ blocks, then it must meet each block in exactly one
coordinate.
\end{proof}

We now combine the decoupling estimate above with the block-full structure.

\begin{lemma}\label{lem:blockwise-cloning-q}
Let
\[
[N]=A_1\dot\cup\cdots\dot\cup A_d
\]
be a labelled disjoint partition. Suppose that
\[
P_A(z)
=
\sum_{\substack{i_1\in A_1,\ldots,i_d\in A_d}}
\sum_{r\in\{1,\ldots,q-1\}^d}
b_{i_1,\ldots,i_d;r}\,
z_{i_1}^{r_1}\cdots z_{i_d}^{r_d},
\qquad z\in\TT^N.
\]
Define the decoupled polynomial
\[
\widetilde P_A(y)
=
\sum_{\substack{i_1\in A_1,\ldots,i_d\in A_d}}
\sum_{r\in\{1,\ldots,q-1\}^d}
b_{i_1,\ldots,i_d;r}\,
y_{i_1,r_1}\cdots y_{i_d,r_d}.
\]
Then
\[
\|\widetilde P_A\|_{L_\infty}
\le
k(q)^d\,\|P_A\|_{L_\infty(\TT^N)},
\]
where $k(q)=O(\sqrt q)$ is the constant from
Lemma~\ref{lem:local-cloning-q}. Here the $L_\infty$-norm of
$\widetilde P_A$ is taken over the corresponding product torus in the
variables $y_{i,r}$.
\end{lemma}

\begin{proof}
If one of the blocks $A_k$ is empty, then $\widetilde P_A=0$, and there is nothing to
prove. We may therefore assume that all blocks are non-empty.

For each block $A_k$, set
\[
 X_{A_k}^{(q)}
 :=
 \operatorname{span}\{z_i^r\colon i\in A_k,\ 1\le r\le q-1\}
 \subset C(\TT^{A_k})
\]
and
\[
 Y_{A_k}^{(q)}
 :=
 \operatorname{span}\{y_{i,r}\colon i\in A_k,\ 1\le r\le q-1\}
 \subset C(\TT^{A_k\times\{1,\ldots,q-1\}}).
\]
By Lemma~\ref{lem:local-cloning-q}, the map
\[
 U_{A_k}\colon X_{A_k}^{(q)}\to Y_{A_k}^{(q)},
 \quad
 U_{A_k}(z_i^r)=y_{i,r},
\]
satisfies
\[
 \|U_{A_k}\|\le k(q).
\]

Since $P_A$ contains exactly one active variable from each block, it belongs to
the algebraic tensor product
\[
 X_{A_1}^{(q)}
 \otimes\cdots\otimes
 X_{A_d}^{(q)}.
\]
We equip this tensor product with the injective tensor norm. Recall that, for
Banach spaces $X_1,\ldots,X_d$, the injective tensor norm on
$X_1\otimes\cdots\otimes X_d$ is defined by
\[
\varepsilon(u)
=
\sup\big\{
\big|
(x_1^*\otimes\cdots\otimes x_d^*)(u)
\big|
\colon
x_j^*\in B_{X_j^*}
\big\}.
\]
We refer to \cite{defant1992tensor} for background on injective tensor
products and their relation with $C(K)$-spaces.

The canonical inclusions
\[
 X_{A_k}^{(q)}\hookrightarrow C(\TT^{A_k})
\]
are isometries. Hence, by the injectivity of the injective tensor norm, the
induced map
\[
 X_{A_1}^{(q)}\otimes_\varepsilon\cdots
 \otimes_\varepsilon X_{A_d}^{(q)}
 \hookrightarrow
 C(\TT^{A_1})\otimes_\varepsilon\cdots
 \otimes_\varepsilon C(\TT^{A_d})
\]
is also an isometry. Using the standard isometric identification
\[
 C(\TT^{A_1})\widehat{\otimes}_\varepsilon\cdots
 \widehat{\otimes}_\varepsilon C(\TT^{A_d})
 =
 C(\TT^{A_1}\times\cdots\times\TT^{A_d}),
\]
we obtain
\[
 \|P_A\|_{X_{A_1}^{(q)}\otimes_\varepsilon\cdots
 \otimes_\varepsilon X_{A_d}^{(q)}}
 =
 \|P_A\|_{L_\infty(\TT^{A_1}\times\cdots\times\TT^{A_d})}.
\]
Under the natural identification
\[
 \TT^{A_1}\times\cdots\times\TT^{A_d}
 =
 \TT^N,
\]
the right-hand side is $\|P_A\|_{L_\infty(\TT^N)}$.

Similarly, the canonical inclusions
\[
 Y_{A_k}^{(q)}
 \hookrightarrow
 C(\TT^{A_k\times\{1,\ldots,q-1\}})
\]
are isometries, and therefore the injective tensor norm on
\[
 Y_{A_1}^{(q)}\otimes\cdots\otimes Y_{A_d}^{(q)}
\]
coincides with the corresponding supremum norm on the product torus. Moreover,
\[
 \widetilde P_A
= ( U_{A_1}\otimes\cdots\otimes U_{A_d} ) P_A.
\]
By the metric mapping property of the injective tensor norm,
\[
 \|\widetilde P_A\|_{L_\infty}
 \le
 \prod_{k=1}^d\|U_{A_k}\|\,
 \|P_A\|_{L_\infty(\TT^N)}
 \le
 k(q)^d\|P_A\|_{L_\infty(\TT^N)}.
\]
This completes the proof.
\end{proof}

We now combine the preceding ingredients to prove the spherical
Bohnenblust--Hille inequality.

\begin{proof}[Proof of Theorem~\ref{thm:BH-toric-spherical}]
Set $p_d:=\frac{2d}{d+1}$. Let $P\in \mathcal B_{d,q}(\TT^N)$ be given by
\[
 P(z)
 =
 \sum_{\alpha\in\Lambda_{d,q,N}} a_\alpha z^\alpha,
 \quad z\in\TT^N.
\]

Fix a labelled partition
\[
 [N]=A_1\dot\cup\cdots\dot\cup A_d.
\]
Let $P_A$ be the block-full part of $P$, that is, the sum of all monomials
whose support meets every block of the partition. By
Lemma~\ref{lem:block-full-projection-q},
\[
 \|P_A\|_\infty
 \le
 2^d\|P\|_\infty.
\]
Moreover, every monomial in $P_A$ meets each block in exactly one coordinate.
Hence $P_A$ can be written in the form
\[
 P_A(z)
 =
 \sum_{\substack{i_1\in A_1,\ldots,i_d\in A_d}}
 \sum_{r\in\{1,\ldots,q-1\}^d}
 b_{i_1,\ldots,i_d;r}
 z_{i_1}^{r_1}\cdots z_{i_d}^{r_d},
 \quad z\in\TT^N.
\]
Here the coefficients $b_{i_1,\ldots,i_d;r}$ are precisely the coefficients
$a_\alpha$ whose support is block-full, written according to the labelled
blocks.

Define the cloned polynomial
\[
 \widetilde P_A(y)
 =
 \sum_{\substack{i_1\in A_1,\ldots,i_d\in A_d}}
 \sum_{r\in\{1,\ldots,q-1\}^d}
 b_{i_1,\ldots,i_d;r}
 y_{i_1,r_1}\cdots y_{i_d,r_d}.
\]
By Lemma~\ref{lem:blockwise-cloning-q},
\[
 \|\widetilde P_A\|_\infty
 \le
 k(q)^d\|P_A\|_\infty
 \le
 (2k(q))^d\|P\|_\infty.
\]
The polynomial $\widetilde P_A$ is tetrahedral and $d$-linear with respect to
the $d$ labelled blocks of variables. Therefore the multilinear
Bohnenblust--Hille inequality gives
\[
 \bigg(
 \sum_{\substack{i_1\in A_1,\ldots,i_d\in A_d}}
 \sum_{r\in\{1,\ldots,q-1\}^d}
 |b_{i_1,\ldots,i_d;r}|^{p_d}
 \bigg)^{1/p_d}
 \le
 \mathrm{BH}_d\,
 \|\widetilde P_A\|_\infty.
\]
Consequently,
\begin{equation}\label{eq:BH-A}
 \bigg(
 \sum_{\substack{i_1\in A_1,\ldots,i_d\in A_d}}
 \sum_{r\in\{1,\ldots,q-1\}^d}
 |b_{i_1,\ldots,i_d;r}|^{p_d}
 \bigg)^{1/p_d}
 \le
 \mathrm{BH}_d(2k(q))^d
 \|P\|_\infty.
\end{equation}

We now choose the labelled partition randomly. Let $\Omega:=\{1,\ldots,d\}^{[N]}$ and equip it with the 
uniform probability measure. Thus each $\omega\in\Omega$ is a map
\[
 \omega\colon [N]\to\{1,\ldots,d\}.
\]
For $\omega\in\Omega$, define
\[
 A_k(\omega):=\{i\in[N]\colon \omega(i)=k\},
 \quad k=1,\ldots,d.
\]
Then
\[
 A(\omega):=(A_1(\omega),\ldots,A_d(\omega))
\]
is a labelled partition of $[N]$. Equivalently, the random choice of $\omega$
assigns each coordinate $i\in[N]$, independently of all other coordinates, to
one of the $d$ labelled blocks, each with probability $1/d$.

Let $S\subset[N]$ be fixed with $|S|=d$. Restricting $\omega$ to $S$, there
are $d^d$ possible maps
\[
 \omega|_S\colon S\to\{1,\ldots,d\},
\]
and all of them occur with the same probability. The set $S$ meets every block
of the partition $A(\omega)$ if and only if
\[
 S\cap A_k(\omega)\neq\varnothing
 \quad\text{for every } k=1,\ldots,d.
\]
Since $|S|=d$, this is equivalent to saying that $S$ meets every block exactly
once. In terms of $\omega|_S$, this means precisely that
\[
 \omega|_S\colon S\to\{1,\ldots,d\}
\]
is a bijection. There are $d!$ such bijections. Hence
\[
 \mathbb P\big(
 S\text{ meets every block of }A(\omega)
 \big)
 =
 \frac{d!}{d^d}.
\]

For fixed $\alpha\in\Lambda_{d,q,N}$, define
\[
 \mathbf 1_\alpha(\omega)
 :=
 \mathbf 1_{\{\operatorname{supp}(\alpha)
      \text{ meets every block of }A(\omega)\}}.
\]
Applying the computation above to $S=\operatorname{supp}(\alpha)$, we obtain
\[
 \mathbb E_\omega \mathbf 1_\alpha(\omega)
 =
 \frac{d!}{d^d}.
\]

Now define
\[
 X(\omega)
 :=
 \sum_{\alpha\in\Lambda_{d,q,N}}
 \mathbf 1_\alpha(\omega)|a_\alpha|^{p_d}.
\]
Thus $X(\omega)$ is the $p_d$-coefficient mass captured by the labelled
partition $A(\omega)$. By linearity of expectation,
\[
\begin{aligned}
 \mathbb E_\omega X(\omega)
 =
 \mathbb E_\omega
 \sum_{\alpha\in\Lambda_{d,q,N}}
 \mathbf 1_\alpha(\omega)|a_\alpha|^{p_d} &=
 \sum_{\alpha\in\Lambda_{d,q,N}}
 |a_\alpha|^{p_d}
 \mathbb E_\omega \mathbf 1_\alpha(\omega) =
 \frac{d!}{d^d}
 \sum_{\alpha\in\Lambda_{d,q,N}}
 |a_\alpha|^{p_d}.
\end{aligned}
\]
Consequently, there exists $\omega_0\in\Omega$ such that
\[
 X(\omega_0)\ge \mathbb E_\omega X(\omega).
\]
For the corresponding labelled partition
\[
 A(\omega_0)=(A_1(\omega_0),\ldots,A_d(\omega_0)),
\]
which we simply denote by
\[
 [N]=A_1\dot\cup\cdots\dot\cup A_d,
\]
we obtain
\[
 \sum_{\substack{\alpha\in\Lambda_{d,q,N}\\
 \operatorname{supp}(\alpha)\text{ meets every block}}}
 |a_\alpha|^{p_d}
 \ge
 \frac{d!}{d^d}
 \sum_{\alpha\in\Lambda_{d,q,N}}
 |a_\alpha|^{p_d}.
\]
For this partition, the estimate \eqref{eq:BH-A} gives
\[
 \bigg(
 \frac{d!}{d^d}
 \sum_{\alpha\in\Lambda_{d,q,N}}
 |a_\alpha|^{p_d}
 \bigg)^{1/p_d}
 \le
 \mathrm{BH}_d(2k(q))^d
 \|P\|_\infty.
\]
Therefore
\[
 \bigg(
 \sum_{\alpha\in\Lambda_{d,q,N}}
 |a_\alpha|^{p_d}
 \bigg)^{1/p_d}
 \le
 \left(\frac{d^d}{d!}\right)^{1/p_d}
 \mathrm{BH}_d(2k(q))^d
 \|P\|_\infty.
\]

It remains to estimate the constants. Since $p_d=\frac{2d}{d+1}\ge1$, we have
\[
 \left(\frac{d^d}{d!}\right)^{1/p_d}
 \le
 \frac{d^d}{d!}
 \le
 e^d,
\]
where the last inequality follows from $d!\ge(d/e)^d$. Moreover, the
multilinear Bohnenblust--Hille constants are known to have subexponential
growth, and in fact even polynomial growth; see
\cite[Proposition~6.18]{defant2019dirichlet}. Polynomial estimates of this
type were first obtained by Bayart, Pellegrino, and Seoane-Sep\'ulveda
\cite{bayart2014bohr}, and were already implicit in earlier ideas from
\cite{defant2010coordinatewise}. In particular, there exists an absolute
constant $B\ge1$ such that, for every $d\ge1$,
\[
 \mathrm{BH}_d\le B^d.
\]
Thus
\[
 \bigg(
 \sum_{\alpha\in\Lambda_{d,q,N}}
 |a_\alpha|^{p_d}
 \bigg)^{1/p_d}
 \le
 \big(2ek(q)B\big)^d
 \|P\|_\infty.
\]
The assertion follows with $c(q):=2ek(q)B=O(\sqrt{q})$.
\end{proof}

\medskip

Note that the finite alphabet $\{1,\ldots,q-1\}$ enters
only through the local decoupling constant $k(q)$, and therefore affects only the
base $c(q)$, not the exponential dependence on $d$. This is the mechanism that
prevents the factorial loss which would arise from full symmetrization.

Combining Theorem~\ref{thm:BH-toric-spherical} with
Theorem~\ref{thm:remez-support}, we obtain a Bohnenblust--Hille inequality
for the spherical space of level $d$ on the Hamming scheme.

\begin{theorem}\label{BH cyclic}
For every $q\ge 2$ there is a constant
$D(q)=O\!\big(\sqrt q\,\log q\big)$ such that, for every
$d,N\in\mathbb N$ and every function
\[
f
=
\sum_{\substack{\bm m\in\mathbb Z_q^N\\
|\operatorname{supp}(\bm m)|= d}}
\widehat f(\bm m)\chi_{\bm m}
\in
\mathcal B_d(C_q^N),
\]
one has
\[
 \Bigg(
 \sum_{\substack{\bm m\in\mathbb Z_q^N\\
|\operatorname{supp}(\bm m)|= d}}
 |\widehat f(\bm m)|^{\frac{2d}{d+1}}
 \Bigg)^{\frac{d+1}{2d}}
 \le
 D(q)^d
 \|f\|_{L_\infty(C_q^N)}.
\]
Moreover, the exponent $\frac{2d}{d+1}$ is optimal (see Remark~\ref{optimal exponent}).
\end{theorem}

\subsection{The low-level regime}

We now pass from exact levels to spaces of level at most $d$. Recall
\[
\mathcal B_{\le d,q}(\TT^N)
:=
\operatorname{span}\{z^\alpha:\alpha\in\{0,\ldots,q-1\}^N,\ 
|\operatorname{supp}(\alpha)|\le d\}.
\]
We shall use two classical one-dimensional estimates: the Remez inequality
\cite{Remez} and Markov's coefficient estimate \cite{Markov}; for background
on these inequalities and on the basic properties of Chebyshev polynomials, we
also refer to \cite{BorweinErdelyi1995}.

Let $E\subset[-1,1]$ be a measurable set of positive Lebesgue measure. If
\[
p(x)=\sum_{j=0}^d a_jx^j
\]
is a polynomial with complex coefficients of degree at most $d$, then
\[
\|p\|_{C([-1,1])} \leq T_d\bigg(\frac{4}{|E|}-1\bigg)\|p\|_{C(E)},
\]
where $T_d$ denotes the Chebyshev polynomial of degree $d$, characterized by
\[
T_d(\cos\theta)=\cos(d\theta), \qquad \theta \in \mathbb{R}. 
\]
The real-valued case is the classical Remez inequality, and the complex-valued
case follows by applying the real-valued estimate to a suitable rotation of the
real part of $p$.

Taking $E=[0,1]$, we obtain
\[
\|p\|_{C([-1,1])} \leq T_d(3)\|p\|_{C([0,1])}.
\]
For $x\geq 1$, we shall use the standard formula
\begin{equation}\label{eq:chebyshev-standard-formula}
T_d(x)
=
\frac{1}{2}
\bigg[\bigl(x+\sqrt{x^2-1}\bigr)^d
+\bigl(x-\sqrt{x^2-1}\bigr)^d
\bigg].
\end{equation}
Thus
\[
T_d(3)
=
\frac{1}{2}\bigl((3+2\sqrt2)^d+(3-2\sqrt2)^d\bigr)
=
\frac{1}{2}\bigl((1+\sqrt2)^{2d}+(1-\sqrt2)^{2d}\bigr).
\]

We shall also use Markov's coefficient estimate. Write
\[
T_m(x)=\sum_{j=0}^m t_j^{(m)}x^j.
\]
For a polynomial $p(x)=\sum_{j=0}^d a_jx^j$, one has
\[
|a_k| \leq B_k^{(d)}\|p\|_{C([-1,1])}, \qquad 1\leq k\leq d,
\]
where
\[
B_k^{(d)}
=
\begin{cases}
|t_k^{(d)}|, & k\equiv d \pmod 2,\\
|t_k^{(d-1)}|, & k\equiv d-1 \pmod 2.
\end{cases}
\]
For complex-valued polynomials this follows from the real-valued case by the same rotation argument as above. 
For our purposes it is enough to note that
\[
B_k^{(d)} \leq (1+\sqrt2)^d.
\]
Indeed, the nonzero coefficients of $T_m$ have the same parity as $m$ and
alternate in sign. Hence, after substituting $x=i$, all nonzero terms have the same argument, and therefore
\[
\sum_{j=0}^m |t_j^{(m)}|
=
|T_m(i)|.
\]
Using the standard complex extension of
\eqref{eq:chebyshev-standard-formula}, we get
\[
|T_m(i)|=\frac{1}{2} \left|\bigl(i+\sqrt{i^2-1}\bigr)^m
+\bigl(i-\sqrt{i^2-1}\bigr)^m
\right| \leq (1+\sqrt2)^m.
\]
Consequently $B_k^{(d)}\leq (1+\sqrt2)^d$.

Combining the Remez estimate on $[0,1]$ with Markov's coefficient estimate gives, for $1\leq k\leq d$,
\[
|a_k| \leq (1+\sqrt2)^d T_d(3)\|p\|_{C([0,1])}.
\]
Therefore
\begin{equation} \label{eq:markov}
    |a_k| \leq \frac{1}{2}(1+\sqrt2)^d
\bigl((1+\sqrt2)^{2d}+(1-\sqrt2)^{2d}\bigr) \|p\|_{C([0,1])}.
\end{equation}
In particular,
\[
|a_k| < (1+\sqrt2)^{3d}\|p\|_{C([0,1])},
\qquad 1\leq k\leq d.
\]
For $k=0$, the simpler estimate
\[
|a_0| \leq
\|p\|_{C([0,1])}
\]
is immediate.

The key point of the following result is that the Fourier projection onto each spherical level has norm at most exponential in $d$, uniformly in $N$ and $q$.

\begin{proposition}\label{prop:level-projection-low}
For each $0\leq k\leq d$, the norm of the Fourier projection
\[
\Pi_k \colon \mathcal B_{\le d,q}(\TT^N)
\longrightarrow \mathcal B_{k,q}(\TT^N)
\]
onto the spherical level $k$ satisfies $\|\Pi_k\|\leq C(d)$, where
\[
C(d):=\frac{1}{2}(1+\sqrt2)^d
\bigl((1+\sqrt2)^{2d}+(1-\sqrt2)^{2d}\bigr)
< (1+\sqrt2)^{3d}.
\]
\end{proposition}

\begin{proof}
For each coordinate $j=1,\ldots,N$, let
\[
E_j:
L_\infty(\TT^N)\longrightarrow L_\infty(\TT^N)
\]
denote the averaging operator in the $j$-th variable:
\[
E_jf(z)
=
\int_{\TT}
f(z_1,\ldots,z_{j-1},\omega,z_{j+1},\ldots,z_N)\,d\omega,
\]
where $d\omega$ denotes the normalized Haar measure on $\TT$.

Clearly,
\[
\|E_jf\|_{L_\infty(\TT^N)}
\leq
\|f\|_{L_\infty(\TT^N)},
\]
so $E_j$ is a contraction.

Now let
\[
z^\alpha=z_1^{\alpha_1}\cdots z_N^{\alpha_N}
\]
be a monomial. Since Haar integration annihilates nonconstant characters on
$\TT$, we have
\[
\int_{\TT}\omega^m\,d\omega
=
0
\qquad (m\neq0).
\]
Therefore
\[
E_j(z^\alpha)
=
\begin{cases}
z^\alpha, & \alpha_j=0,\\
0, & \alpha_j\neq0.
\end{cases}
\]

Fix $0\le t\le1$, and define
\[
S_t^{(j)}
=
t\,\mathrm{Id}+(1-t)E_j.
\]
Since $S_t^{(j)}$ is a convex combination of contractions, it is itself a
contraction on $L_\infty(\TT^N)$.

Moreover,
\[
S_t^{(j)}(z^\alpha)
=
\begin{cases}
z^\alpha, & \alpha_j=0,\\
t\,z^\alpha, & \alpha_j\neq0.
\end{cases}
\]
Now define
\[
S_t=S_t^{(1)}\cdots S_t^{(N)}.
\]
Since the operators commute and are contractive, $S_t$ is also a contraction.
Applying the preceding identity coordinate by coordinate, we obtain
\[
S_t(z^\alpha)
=
t^{|\operatorname{supp}(\alpha)|}z^\alpha.
\]

Let
\[
P=\sum_{s=0}^d P_s,
\qquad
P_s\in\mathcal B_{s,q}(\TT^N).
\]
Then
\[
S_tP
=
\sum_{s=0}^d t^sP_s.
\]

Fix $z\in\TT^N$. The map
\[
p_z(t):=(S_tP)(z)
=
\sum_{s=0}^d P_s(z)t^s
\]
is a one-variable polynomial of degree at most $d$. Since $S_t$ is a
contraction for every $0\le t\le1$, we have
\[
|p_z(t)|
\leq
\|S_tP\|_{L_\infty(\TT^N)}
\leq
\|P\|_{L_\infty(\TT^N)}
\]
for every $0\le t\le1$.

If $k=0$, then
\[
|P_0(z)|
=
|p_z(0)|
\leq
\|P\|_{L_\infty(\TT^N)}
\leq
C(d)\|P\|_{L_\infty(\TT^N)}.
\]
If $1\leq k\leq d$, the one-dimensional coefficient estimate proved in \eqref{eq:markov} gives
\[
|P_k(z)|
\leq
C(d)\|P\|_{L_\infty(\TT^N)}.
\]
Taking the supremum over $z\in\TT^N$, we obtain
\[
\|\Pi_kP\|_{L_\infty(\TT^N)}
=
\|P_k\|_{L_\infty(\TT^N)}
\leq
C(d)\|P\|_{L_\infty(\TT^N)}.
\]
This proves the desired estimate.
\end{proof}

\begin{remark}\label{rem:level-projection-discrete}
An entirely analogous statement holds on the discrete Hamming scheme
$C_q^N$. More precisely, the Fourier projection
\[
\Pi_k:
\mathcal B_{\le d}(C_q^N)
\longrightarrow
\mathcal B_{k}(C_q^N)
\]
also satisfies
\[
\|\Pi_k\|\leq C(d),
\]
where $C(d)$ is the same constant as in Proposition
\ref{prop:level-projection-low}. The proof is identical, replacing Haar
integration on $\TT$ by averaging over $C_q$ and monomials $z^\alpha$ by
Fourier characters $\chi_{\bm m}$.
\end{remark}

We can now extend the exact-level Bohnenblust--Hille inequalities to spaces
of level at most $d$.

\begin{theorem}[Support-sensitive Bohnenblust--Hille inequality]\label{thm:BH-toric-spherical-low}
For every $q\ge2$ there exists a constant $c_1(q)=O(\sqrt q)$ such that, for every
$d,N\in\mathbb N$ and every
\[
P
=
\sum_{\substack{\alpha\in\{0,\ldots,q-1\}^N\\
|\operatorname{supp}(\alpha)|\le d}}
a_\alpha z^\alpha
\in
\mathcal B_{\le d,q}(\TT^N),
\]
one has
\[
\Bigg(
\sum_{\substack{\alpha\in\{0,\ldots,q-1\}^N\\
|\operatorname{supp}(\alpha)|\le d}}
|a_\alpha|^{\frac{2d}{d+1}}
\Bigg)^{\frac{d+1}{2d}}
\le
c_1(q)^d
\|P\|_{L_\infty(\TT^N)}.
\]
\end{theorem}

\begin{proof}
Write
\[
P=\sum_{k=0}^d P_k,
\qquad P_k\in\mathcal B_{k,q}(\TT^N).
\]
By Proposition~\ref{prop:level-projection-low},
\[
\|P_k\|_{L_\infty(\TT^N)}
\le c^d\|P\|_{L_\infty(\TT^N)}.
\]
For $k\ge1$, applying Theorem~\ref{thm:BH-toric-spherical} to $P_k$ and using
that $\ell_{2k/(k+1)}$ dominates $\ell_{2d/(d+1)}$, we get
\[
\bigg(
\sum_{|\operatorname{supp}(\alpha)|=k}
|a_\alpha|^{\frac{2d}{d+1}}
\bigg)^{\frac{d+1}{2d}}
\le
c(q)^k \|P_k\|_{L_\infty(\TT^N)}
\le
c(q)^k c^d\|P\|_{L_\infty(\TT^N)}
\le
(c(q) c)^d\|P\|_{L_\infty(\TT^N)}.
\]
The level $k=0$ is trivial. Hence,
\[
\sum_{|\operatorname{supp}(\alpha)|\le d}
|a_\alpha|^{\frac{2d}{d+1}}
\le
d(c(q) c)^{d\frac{2d}{d+1}}\|P\|_{L_\infty(\TT^N)}^{\frac{2d}{d+1}}
\le
(2c(q) c)^{d\frac{2d}{d+1}}\|P\|_{L_\infty(\TT^N)}^{\frac{2d}{d+1}}.
\]
This concludes the proof.
\end{proof}

% \begin{proof}
% Write
% \[
% P=\sum_{k=0}^d P_k,
% \qquad P_k\in\mathcal B_{k,q}(\TT^N).
% \]
% By Proposition~\ref{prop:level-projection-low},
% \[
% \|P_k\|_{L_\infty(\TT^N)}
% \le c^d\|P\|_{L_\infty(\TT^N)}.
% \]
% For $k\ge1$, applying Theorem~\ref{thm:BH-toric-spherical} to $P_k$ and using
% that $\ell_{2k/(k+1)}$ dominates $\ell_{2d/(d+1)}$, we get
% \[
% \bigg(
% \sum_{|\operatorname{supp}(\alpha)|=k}
% |a_\alpha|^{\frac{2d}{d+1}}
% \bigg)^{\frac{d+1}{2d}}
% \le
% c(q)^k \|P_k\|_{L_\infty(\TT^N)}
% \le
% c(q)^k c^d\|P\|_{L_\infty(\TT^N)}.
% \]
% The level $k=0$ is trivial. Hence, setting
% $p=2d/(d+1)$ and
% \[
% A_k=
% \bigg(
% \sum_{|\operatorname{supp}(\alpha)|=k}
% |a_\alpha|^p
% \bigg)^{1/p},
% \]
% we have, since the levels are disjoint,
% \[
% \bigg(
% \sum_{|\operatorname{supp}(\alpha)|\le d}
% |a_\alpha|^p
% \bigg)^{1/p}
% =
% \bigg(
% \sum_{k=0}^d A_k^p
% \bigg)^{1/p}
% \le
% \sum_{k=0}^d A_k .
% \]
% Using the preceding estimate level by level, we obtain
% \[
% \bigg(
% \sum_{|\operatorname{supp}(\alpha)|\le d}
% |a_\alpha|^p
% \bigg)^{1/p}
% \le
% c^d\|P\|_{L_\infty(\TT^N)}
% \sum_{k=0}^d \max\{1,c(q)\}^k.
% \]
% Since
% \[
% \sum_{k=0}^d \max\{1,c(q)\}^k
% \le
% (d+1)\max\{1,c(q)\}^d
% \le
% \big(2\max\{1,c(q)\}\big)^d,
% \]
% the desired estimate follows, after changing the exponential constant.
% \end{proof}
Combining this with the Remez-type transfer gives the corresponding
inequality on the Hamming scheme.

\begin{theorem}\label{BH cyclic low}
For every $q\ge2$ there exists a constant
$D_1(q)=O\!\big(\sqrt q\,\log q\big)$ such that, for every
$d,N\in\mathbb N$ and every function
\[
f
=
\sum_{\substack{\bm m\in\mathbb Z_q^N\\
|\operatorname{supp}(\bm m)|\le d}}
\widehat f(\bm m)\chi_{\bm m}
\in
\mathcal B_{\le d}(C_q^N),
\]
one has
\[
\Bigg(
\sum_{\substack{\bm m\in\mathbb Z_q^N\\
|\operatorname{supp}(\bm m)|\le d}}
|\widehat f(\bm m)|^{\frac{2d}{d+1}}
\Bigg)^{\frac{d+1}{2d}}
\le
D_1(q)^d
\|f\|_{L_\infty(C_q^N)}.
\]
Moreover, the exponent $\frac{2d}{d+1}$ is optimal.
\end{theorem}

\begin{proof}
The proof is identical to the proof of Theorem~\ref{thm:BH-toric-spherical-low},
using Remark~\ref{rem:level-projection-discrete} and the
exact-level estimate of Theorem~\ref{BH cyclic}. Equivalently, one may first
apply the toric estimate to the analytic extension of $f$ and then transfer
back to $C_q^N$ by the Remez-type theorem.
The optimality follows from the exact-level case, since
\(
\mathcal B_d(C_q^N)\subset \mathcal  B_{\le d}(C_q^N).
\)
\end{proof}

\subsection{Learning low-level functions on Hamming schemes}

Broadly speaking, learning theory studies the problem of reconstructing an
unknown function from random observations. A central question is to determine
how many samples are needed to achieve a prescribed approximation accuracy.
The support-sensitive Bohnenblust--Hille inequalities established in the
previous section provide precisely the type of coefficient control needed to
address this question in our setting.

When $q=2$, the support level of a frequency coincides with its total degree,
so the spaces considered here reduce to the classical low-degree setting
studied by Eskenazis and Ivanisvili~\cite{eskenazis2022learning}. For
$q\ge3$, however, bounded-level functions form a strictly larger class. Indeed,
every polynomial of degree at most $d$ belongs to
$\mathcal B_{\le d}(C_q^N)$, while a function in
$\mathcal B_{\le d}(C_q^N)$ may have total degree as large as
$(q-1)d$. Thus our framework extends the low-degree setting considered in
\cite{becker2025dimension}.

Let $\Lambda_{{\le d},q,N}$ denote the set of frequencies supported on at most
$d$ coordinates, whose cardinality is
$|\Lambda_{{\le d},q,N}|=\sum_{k=0}^d (q-1)^k\binom Nk$.

\begin{theorem}\label{learning spherical}
Let $q\ge2$, $d\ge1$, $N\ge1$, and let
$f\in\mathcal B_{\le d}(C_q^N)$ satisfy
$\|f\|_{L_\infty(C_q^N)}\le1$. Let $0<\varepsilon,\delta<1$.

Then there exists a randomized algorithm which, from $M$ independent
uniformly distributed samples
$(X_1,f(X_1)),\ldots,(X_M,f(X_M))$, with
$X_j\sim\mathrm{Unif}(C_q^N)$, constructs a random hypothesis
$h:C_q^N\to\mathbb C$ such that
\[
\|h-f\|_{L^2(C_q^N)}^2<\varepsilon
\]
with probability at least $1-\delta$, provided that
\[
M
\ge
C_0\, d\, D_1(q)^{2d^2}\,
\varepsilon^{-(d+1)}
\log\!\left(
\frac{|\Lambda_{{\le d},q,N}|}{\delta}
\right),
\]
where $C_0>0$ is a universal constant and
$D_1(q)=O\!\big(\sqrt q\,\log q\big)$. 
\end{theorem}

\begin{proof}
The proof follows the argument of \cite{eskenazis2022learning} almost verbatim, the main new ingredient being the support-sensitive Bohnenblust--Hille inequality established in the previous section.

Set $L:=|\Lambda_{{\le d},q,N}|$ and put $p:=2d/(d+1)$, thus
$2-p=2/(d+1)$.

Let $X_1,\ldots,X_M$ be independent random variables uniformly distributed
on $C_q^N$. For each $\bm m\in\Lambda_{{\le d},q,N}$, define
\[
\widetilde f(\bm m)
:=
\frac1M\sum_{j=1}^M
f(X_j)\overline{\chi_{\bm m}(X_j)}.
\]
Since $\widehat f(\bm m)$ is the expectation of
$f(X)\overline{\chi_{\bm m}(X)}$ with respect to the uniform measure on
$C_q^N$, we have $\mathbb E\,\widetilde f(\bm m)=\widehat f(\bm m)$.

We first prove simultaneous concentration. Since $\|f\|_\infty\le1$ and
$|\chi_{\bm m}|=1$, by the classical Hoeffding's inequality applied to the real and imaginary parts,
\[
\mathbb P\left(
|\widetilde f(\bm m)-\widehat f(\bm m)|>b
\right)
\le
4e^{-cMb^2},
\]
for some universal constant $c>0$.

Hence, by the union bound, the event
\[
G_b:=
\left\{
|\widetilde f(\bm m)-\widehat f(\bm m)|\le b
\text{ for every } \bm m\in\Lambda_{{\le d},q,N}
\right\}
\]
has probability at least \(1-\delta\) whenever $
M\ge C b^{-2}\log(4L/\delta),$
where \(C>0\) is a universal constant.

Fix $0<b<a$ and define
\[
\mathcal A
:=
\left\{
\bm m\in\Lambda_{{\le d},q,N}:
|\widetilde f(\bm m)|\ge a
\right\},
\qquad
h(x):=
\sum_{\bm m\in\mathcal A}
\widetilde f(\bm m)\chi_{\bm m}(x).
\]
Assume that $G_b$ holds. If $\bm m\in\mathcal A$, then
$|\widehat f(\bm m)|\ge a-b$. Therefore, by Theorem~\ref{BH cyclic low},
\[
|\mathcal A|(a-b)^p
\le
\sum_{\bm m\in\mathcal A}|\widehat f(\bm m)|^p
\le
D_1(q)^{dp},
\]
and so $|\mathcal A|\le D_1(q)^{dp}(a-b)^{-p}$.

On the other hand, if $\bm m\notin\mathcal A$, then
$|\widehat f(\bm m)|<a+b$. By Parseval's identity,
\[
\|h-f\|_{L^2(C_q^N)}^2
=
\sum_{\bm m\in\mathcal A}
|\widetilde f(\bm m)-\widehat f(\bm m)|^2
+
\sum_{\bm m\notin\mathcal A}
|\widehat f(\bm m)|^2 .
\]
The first term is bounded by $D_1(q)^{dp}(a-b)^{-p}b^2$. For the second one,
we use
$$|\widehat f(\bm m)|^2
\le
(a+b)^{2-p}|\widehat f(\bm m)|^p$$
for $\bm m\notin\mathcal A$, and again Theorem~\ref{BH cyclic low}. Hence
\[
\sum_{\bm m\notin\mathcal A}
|\widehat f(\bm m)|^2
\le
D_1(q)^{dp}(a+b)^{2-p}.
\]
Thus
\[
\|h-f\|_{L^2(C_q^N)}^2
\le
D_1(q)^{dp}
\left[
(a-b)^{-p}b^2+(a+b)^{2-p}
\right].
\]

Choose $a=b(1+\sqrt{d+1})$. Then an elementary estimate gives
\[
\|h-f\|_{L^2(C_q^N)}^2
\le
D_1(q)^{dp}
b^{2/(d+1)}
\bigl(e^4(d+1)\bigr)^{1/(d+1)}.
\]
Now take
\[
b^2
=
\frac{\varepsilon^{d+1}}
{e^4(d+1)\,D_1(q)^{dp(d+1)}}.
\]
Since $dp(d+1)=2d^2$, this choice gives
$\|h-f\|_{L^2(C_q^N)}^2\le\varepsilon$.

Finally, substituting this value of $b$ into
$M\ge C b^{-2}\log(4L/\delta)$, we obtain
\[
M
\ge
C_0\, d\,
D_1(q)^{2d^2}
\varepsilon^{-(d+1)}
\log\left(\frac{ L}{\delta}\right),
\]
for a universal constant $C_0>0$. 
\end{proof}

The preceding theorem should be compared with the low-degree learning result
of Klein, Slote, Volberg, and Zhang \cite{klein2024quantum}. Since every frequency supported on at most $d$
coordinates has total degree at most $(q-1)d$, their theorem also applies to
$\mathcal B_{\le d}(C_q^N)$. However, our support-sensitive approach improves
the dependence on the accuracy parameter from
$\varepsilon^{-((q-1)d+1)}$ to $\varepsilon^{-(d+1)}$, while preserving the
same logarithmic dependence on the ambient dimension $N$.
\section{Dimension-free comparison of local invariants}\label{sec: local}

We begin the proof of the main results with the local-invariant part of the
article. The aim of this section is to compare, in a dimension-free way, the
Sidon constant, the unconditional basis constant, and the Gordon--Lewis
constant for the polynomial spaces introduced above. The key point is that, for
fixed degree, these invariants have the same asymptotic order up to constants
depending only on $q$ and exponentially on $d$, but not on the ambient
dimension $N$.

The argument proceeds through a transfer from the finite group $C_q^N$ to the
torus $\TT^N$. This allows us to use polynomial and Banach-space techniques on
analytic polynomial spaces and then return to the Hamming scheme by means of
Remez-type inequalities.

\noindent{\bf The unconditional basis constant.}
Let $(e_i)_{i\in I}$ be a basis of a Banach space $X$. Its unconditional
basis constant is the infimum over all $K>0$ such that, for every finitely
supported family of scalars $(\alpha_i)_{i\in I}$ and every finitely supported
family of unimodular scalars $(\varepsilon_i)_{i\in I}$, one has
\[
\bigg\Vert \sum_{i\in I}\varepsilon_i\alpha_i e_i\bigg\Vert
\le
K
\bigg\Vert \sum_{i\in I}\alpha_i e_i\bigg\Vert .
\]
We denote this constant by
$\boldsymbol{\chi}((e_i)_{i\in I})$, or more explicitly by
$\boldsymbol{\chi}((e_i)_{i\in I};X)$ when the ambient space needs to be
specified. If $(e_i)_{i\in I}$ is not unconditional, we set
$\boldsymbol{\chi}((e_i)_{i\in I})=+\infty$. The basis is called
$1$-unconditional whenever
$\boldsymbol{\chi}((e_i)_{i\in I})=1$.

The unconditional basis constant of $X$ is defined by
\[
\boldsymbol{\chi}(X)
:=
\inf \boldsymbol{\chi}((e_i)_{i\in I}),
\]
where the infimum is taken over all unconditional bases of $X$.

For the polynomial spaces considered in this article, we shall use the
canonical basis of characters. For instance, for $\mathcal P_d(C_q^N)$ the
Sidon constant $\sid(\mathcal P_d(C_q^N))$ coincides with the unconditional
basis constant of the character basis
$(\chi_{\bm m})_{|\bm m|=d}$.

\noindent{\bf The Gordon--Lewis constant.}
A significant milestone in the study of unconditional basis constants was
achieved by Gordon and Lewis in \cite{gordon1974absolutely}, who introduced
the constant now known as the Gordon--Lewis constant. Since we shall study
this constant for spaces of polynomials, we first recall some additional
notions needed to present its original definition in the framework of local
Banach space theory.

Let $X$ and $Y$ be Banach spaces. An operator $u\in\mathcal L(X,Y)$ is said
to be $1$-factorable if there exist a measure space $(\Omega,\Sigma,\mu)$ and
operators
\[
v\in\mathcal L(X,L_1(\mu)),
\quad
w\in\mathcal L(L_1(\mu),Y^{\ast\ast})
\]
such that $\kappa_Y u=w v$, or, equivalently, such that the following factorization holds:
\[
\kappa_Y u\colon
X \stackrel{v}{\longrightarrow} L_1(\mu)
\stackrel{w}{\longrightarrow} Y^{\ast\ast}.
\]
Here, as usual, $\kappa_Y\colon Y\to Y^{\ast\ast}$ denotes the canonical
injection. The $\gamma_1$-norm of the $1$-factorable operator $u$ is defined by
\begin{equation}\label{fac-p}
 \gamma_1(u)=\inf \|v\|\,\|w\|,
\end{equation}
where the infimum is taken over all such factorizations.

An operator $u\in\mathcal L(X,Y)$ is said to be $1$-summing if there exists a
constant $c>0$ such that, for every finite choice of vectors
$x_1,\ldots,x_N\in X$, one has
\[
\sum_{j=1}^N \|u x_j\|_Y
\le C
\sup\bigg\{
\sum_{j=1}^{N} |x^\ast(x_j)|\colon
\|x^\ast\|_{X^\ast}\le 1
\bigg\}.
\]
We denote by $\pi_1(u\colon X\to Y)$, or simply by $\pi_1(u)$, the least constant $C$ for which this inequality holds.

A Banach space $X$ has the Gordon-Lewis property if every $1$-summing operator
$u\colon X \to \ell_2$ is $1$-factorable. In this case, there is a~positive number $C$ such that for all $1$-summing operators
$u \colon X \to \ell_2$
\begin{equation*}
\label{eq: def G-L}
\gamma_1(u)\leq C\,\pi_1(u)\,,
\end{equation*}
and the best such $c>0$ is called the Gordon--Lewis constant of $X$ and denoted by $\boldsymbol{g\!l}(X)$.

We can further clarify the aforementioned comment about the crucial role of the Gordon-Lewis constant in the study of unconditionality
in Banach spaces. Specifically, Gordon and Lewis \cite{gordon1974absolutely} showed  that for every Banach space $X$ with
an unconditional basis~$(e_i)_{i \in I}$ the following estimate holds
\begin{equation} \label{gl-inequality}
\boldsymbol{g\!l}(X)\leq  \boldsymbol{\chi}(X) \leq \boldsymbol{\chi}( (e_i)_{i \in I})\,.
\end{equation}
It is worth noting that in contrast the unconditional basis constant, the Gordon-Lewis constant has the useful (ideal)
property that 
\begin{equation}\label{eq: GL inequality}
\boldsymbol{g\!l}(X)\leq \|u\|  \|v\|\,\,\boldsymbol{g\!l}(Y),
\end{equation}
whenever $u\colon X \to Y$ and $v\colon Y \to X$ are operators such that $\text{id}_{X} = uv$.

We will show that the Gordon--Lewis constant, the unconditional basis constant,
and the Sidon constant of the polynomial spaces under consideration are all
equivalent up to dimension-free constants. This equivalence allows us to
determine their asymptotic behaviour.

\subsection{Comparison of Gordon--Lewis, unconditional basis, and  Sidon constants}

Every finite set of multi-indices $J\subset \mathbb N_0^N$ determines a
subspace of analytic polynomials on the torus $\TT^N$, namely
\[
\mathcal P_J(\TT^N)
=
\operatorname{span}\{z^\alpha\colon \alpha\in J\}
\subset C(\TT^N).
\]
For such a space, we denote by
$\boldsymbol{\chi}_{\operatorname{mon}}\big(\mathcal P_J(\TT^N)\big)$
the unconditional basis constant of the canonical monomial basis
$\{z^\alpha\colon \alpha\in J\}$.

Recall that
$\Lambda_{d,q,N}=\{\alpha\in\{0,\ldots,q-1\}^N:
|\operatorname{supp}(\alpha)|=d\}$. In particular, if
$J=\Lambda_{d,q,N}$, then
\[
\mathcal P_{\Lambda_{d,q,N}}(\TT^N)
=
\mathcal B_{d,q}(\TT^N).
\]
Moreover, every $\alpha\in\Lambda_{d,q,N}$ satisfies
\[
|\alpha|
\le
(q-1)|\operatorname{supp}(\alpha)|
=
(q-1)d.
\]

The next result, which is a special case of
\cite[Theorem~2.5]{defant2026local}, shows that for arbitrary sets of
monomials of degree at most $m$, the Gordon--Lewis constant, the unconditional
basis constant, and the monomial unconditional basis constant are equivalent
up to the exponential factor $2^m$. In particular, the result applies not only
to homogeneous spaces but also to spaces of polynomials of degree at most $m$,
which will be important below.

\begin{theorem}\label{thm:gl-versus-unc-torus}
Let $J\subset \mathbb N_0^N$ be a finite set of multi-indices satisfying
$|\alpha|\le m$ for every $\alpha\in J$. Then
\[
\boldsymbol{g\!l}\big(\mathcal P_J(\TT^N)\big)
\le
\boldsymbol{\chi}\big(\mathcal P_J(\TT^N)\big)
\le
\boldsymbol{\chi}_{\operatorname{mon}}\big(\mathcal P_J(\TT^N)\big)
\le
2^m\,
\boldsymbol{g\!l}\big(\mathcal P_J(\TT^N)\big).
\]
\end{theorem}

In particular, the previous theorem applies to
$\mathcal B_{d,q}(\TT^N)$ with $m=(q-1)d$.

Combining Theorem~\ref{thm:gl-versus-unc-torus} with the Remez-type
transfer in Theorem~\ref{thm:remez-support}, and using the ideal property of the
Gordon--Lewis constant, we obtain the following comparison on the finite group
$C_q^N$.

\begin{corollary} \label{gl-versus-unc}
There exists a universal constant $C\ge 1$ such that the following holds.
Let $q\ge2$, $d,N\ge1$, and let
$\Lambda\subseteq \Lambda_{\le d,q,N}$ be non-empty. Set
$m=m(\Lambda):=\max_{\alpha\in\Lambda}|\alpha|$. Then
\begin{equation}\label{eq:relation-for-Lambda}
\boldsymbol{g\!l}\big(\mathcal P_\Lambda(C_q^N)\big)
\le
\boldsymbol{\chi}\big(\mathcal P_\Lambda(C_q^N)\big)
\le
\sid\big(\mathcal P_\Lambda(C_q^N)\big)
\le
(C\log q)^{2d}2^m\,
\boldsymbol{g\!l}\big(\mathcal P_\Lambda(C_q^N)\big).
\end{equation}
In particular, this applies to $\mathcal B_d(C_q^N)$ with
$m=(q-1)d$, and to $\mathcal P_d(C_q^N)$ and
$\mathcal T_d(C_q^N)$ with $m=d$.
\end{corollary}

\begin{proof}
Let $C\ge1$ denote the universal constant appearing in
Theorem~\ref{thm:remez-support}. The inequalities
\[
\boldsymbol{g\!l}(X)\le \boldsymbol{\chi}(X)\le \sid(X)
\]
for the spaces considered here follow from the Gordon--Lewis
inequality~\eqref{gl-inequality} and from the fact that the Sidon constant is
the unconditional basis constant of the corresponding character basis.
Let
\[
R\colon \mathcal P_{\Lambda}(\TT^N)\longrightarrow \mathcal P_{\Lambda}(C_q^N)
\]
be the restriction map. Then $\|R\|\le1$. Its inverse is the canonical
extension map
\[
E\colon \mathcal P_{\Lambda}(C_q^N)\longrightarrow \mathcal P_{\Lambda}(\TT^N),
\]
and Theorem~\ref{thm:remez-support} gives
\[
\|E\|\le (C\log q)^{d}.
\]
Therefore, for every $P\in \mathcal P_d(C_q^N)$, if $EP$ denotes its toric
extension, then
\[
\sum_{\alpha}|\widehat P(\alpha)|
\le
\boldsymbol{\chimon}\big(\mathcal P_{\Lambda}(\TT^N)\big)
\|EP\|_{L_\infty(\TT^N)}
\le
(C\log q)^{d}
\boldsymbol{\chimon}\big(\mathcal P_{\Lambda}(\TT^N)\big)
\|P\|_{L_\infty(C_q^N)}.
\]
Hence
\begin{equation}\label{eq: sid-vs-chimon}
\sid\big(\mathcal P_{\Lambda}(C_q^N)\big)
\le
(C\log q)^{d}
\boldsymbol{\chimon}\big(\mathcal P_{\Lambda}(\TT^N)\big).
\end{equation}
On the other hand, Theorem~\ref{thm:gl-versus-unc-torus} yields
\[
\boldsymbol{\chimon}\big(\mathcal P_{\Lambda}(\TT^N)\big)
\le
2^m\,
\boldsymbol{g\!l}\big(\mathcal P_{\Lambda}(\TT^N)\big).
\]
Moreover, since $ER=\id_{\mathcal P_{\Lambda}(\TT^N)}$, the ideal property
\eqref{eq: GL inequality} gives
\[
\boldsymbol{g\!l}\big(\mathcal P_{\Lambda}(\TT^N)\big)
\le
\|R\|\,\|E\|\,
\boldsymbol{g\!l}\big(\mathcal P_{\Lambda}(C_q^N)\big)
\le
(C\log q)^{d}
\boldsymbol{g\!l}\big(\mathcal P_{\Lambda}(C_q^N)\big).
\]
Combining the last three estimates, we obtain
\[
\sid\big(\mathcal P_{\Lambda}(C_q^N)\big)
\le
2^m(C\log q)^{2d}
\boldsymbol{g\!l}\big(\mathcal P_{\Lambda}(C_q^N)\big).
\]
This concludes the proof.
\end{proof}

\subsection{Asymptotics of Gordon--Lewis, unconditional basis, and  Sidon constants}
In the next result, the notation $A\sim_{c(q)^d}B$ means that there exist constants $c_1(q),c_2(q)>0$, depending only on $q$, such
that
\[
c_1(q)^d B \le A \le c_2(q)^d B.
\]

\begin{theorem}\label{thm:invariants}
Let $\boldsymbol{\eta}$ denote any of the Banach space invariants
$\sid$ $($the Sidon constant$)$, $\boldsymbol{\chi}$ $($the unconditional basis
constant$)$, or $\boldsymbol{g\!l}$ $($the Gordon--Lewis constant$)$. Then, for every
$q\ge2$ and $1\le d\le N$,
\[
\boldsymbol{\eta}\big(\mathcal P_d(C_q^N)\big),
\quad
\boldsymbol{\eta}\big(\mathcal T_d(C_q^N)\big),
\quad
\boldsymbol{\eta}\big(\mathcal B_d(C_q^N)\big)
\sim_{c(q)^d}
\Big(\frac{N}{d}\Big)^{\frac{d-1}{2}}.
\]
where the constants depend only on $q$.
\end{theorem}
For the Sidon constant of $\mathcal P_d(C_q^N)$, estimates of the same order were previously obtained
in \cite[Proposition~3]{SloteVolbergZhangBH} by means of a
Bohnenblust--Hille inequality for cyclic groups. Our methods are different in
nature and treat simultaneously the three invariants $\sid$,
$\boldsymbol{\chi}$, and $\boldsymbol{g\!l}$ using techniques from local Banach
space theory. In both approaches, a key ingredient is a dimension-free
Remez-type inequality.

\begin{proof}
By \cite{defant2011bohnenblust}, see also
\cite[Theorem~2.16]{defant2026local} for an explicit formulation, there exists
a constant $c>0$, independent of $N$ and $d$, such that
\begin{equation}\label{eq:order-of-chimon}
\boldsymbol{\chimon}\big(\mathcal P_{d,q}(\TT^N)\big)
\le
c^d
\bigg(\frac{N}{d}\bigg)^{\frac{d-1}{2}}.
\end{equation}
Combining this estimate with 
\eqref{eq: sid-vs-chimon} 
and
\eqref{eq:relation-for-Lambda} gives the desired upper bound for
$\boldsymbol{\eta}\big(\mathcal P_d(C_q^N)\big)$ and $\boldsymbol{\eta}\big(\mathcal T_d(C_q^N)\big)$.

For the upper bound for spherical spaces, an important
difference appears: the corresponding trigonometric polynomials on $\TT^N$ are
no longer homogeneous of degree $d$, but rather have total degree at most
$(q-1)d$. Because of this, the result from
\cite[Theorem~2.16]{defant2026local} cannot be applied directly, and the
dependence on the degree has to be tracked more carefully.

Recall that
\[
\Lambda_{d,q,N}
=
\big\{
\alpha\in\{0,\ldots,q-1\}^N\colon
|\operatorname{supp}(\alpha)|=d
\big\}, \quad\text{and }\quad |\Lambda_{d,q,N}|
=
{\textstyle\binom{N}{d}}(q-1)^d.
\]
Let $P\in \mathcal B_d(C_q^N)$ be given by
\[
P(z)
=
\sum_{\alpha\in\Lambda_{d,q,N}} a_\alpha z^\alpha,
\quad z\in C_q^N.
\]
By H\"older's inequality with exponents
$\frac{2d}{d+1}$ and $\frac{2d}{d-1}$, with the usual interpretation when
$d=1$, we obtain
\[
\sum_{\alpha\in\Lambda_{d,q,N}} |a_\alpha|
\le
\big({\textstyle\binom{N}{d}}(q-1)^d\big)^{\frac{d-1}{2d}}
\bigg(
\sum_{\alpha\in\Lambda_{d,q,N}}
|a_\alpha|^{\frac{2d}{d+1}}
\bigg)^{\frac{d+1}{2d}}\le
c(q)^d
\bigg(\frac{N}{d}\bigg)^{\frac{d-1}{2}}
\|P\|_{L_\infty(C_q^N)},
\]
where in the last inequality we applied  Theorem~\ref{BH cyclic} together with the estimate $\binom{N}{d}\le (eN/d)^d$.

The desired upper estimates for $\boldsymbol{\chi}$ and $\boldsymbol{g\!l}$
follow from the standard inequalities
\[
\boldsymbol{g\!l}\big(\mathcal B_d(C_q^N)\big)
\le
\boldsymbol{\chi}\big(\mathcal B_d(C_q^N)\big)
\le
\sid\big(\mathcal B_d(C_q^N)\big).
\]

We now turn to the lower bounds. Consider the tetrahedral system
\[
\big(\chi_\alpha\big)_{\alpha\in\{0,1\}^N,\ |\alpha|=d}.
\]
Since $C_q^N\subset\TT^N$, the supremum norm on $\TT^N$ is larger than the
supremum norm on $C_q^N$. Hence
\[
\sid\big(\mathcal T_d(\TT^N)\big)
\le
\sid\big(\mathcal T_d(C_q^N)\big).
\]
Here $\mathcal T_d(\TT^N)$ denotes the space of tetrahedral $d$-homogeneous polynomials, i.e. $\mathcal T_d(\TT^N)=\mathcal P_{d,2}(\TT^N)=\mathcal B_{d,2}(\TT^N)$.
Moreover, the tetrahedral characters form a subsystem of the homogeneous
characters and of the spherical characters, therefore
% \[
% \sid\big(\mathcal T_d(C_q^N)\big)
% \le
% \sid\big(\mathcal P_d(C_q^N)\big).
% \]
%  Thus
\[
\sid\big(\mathcal T_d(\TT^N)\big)
\le
\sid\big(\mathcal T_d(C_q^N)\big)
\le
\sid\big(\mathcal P_d(C_q^N)\big)
\quad\text{and }\quad
\sid\big(\mathcal T_d(\TT^N)\big)
\le
\sid\big(\mathcal T_d(C_q^N)\big)
\le
\sid\big(\mathcal B_d(C_q^N)\big).
\]
The lower bound for $\sid\big(\mathcal T_d(\TT^N)\big)$ follows from
\cite[Theorem~2.16]{defant2026local}, and has order $\big(\frac{N}{d}\big)^{\frac{d-1}{2}}$
up to constants exponential in $d$. This gives the corresponding lower bounds
for the Sidon constants.

Finally, Corollary~\ref{gl-versus-unc} yields the corresponding lower bounds,
up to constants of the form $c(q)^d$, for the unconditional basis constant and
the Gordon--Lewis constant. This completes the proof.
\end{proof}

We conclude with two remarks.

\begin{remark}\label{optimal exponent} We conclude by noting that the exponent in Theorem~\ref{BH cyclic} is
optimal. Indeed, if the exponent $\frac{2d}{d+1}$ could be replaced by some
$p<\frac{2d}{d+1}$, then H\"older's inequality would give an upper bound for
$\sid\big(\mathcal B_d(C_q^N)\big)$ with a strictly smaller power of $N$ than
$\frac{d-1}{2}$. This would contradict the asymptotic estimate obtained in
Theorem~\ref{thm:invariants}.
\end{remark}

\begin{remark}
Whenever Proposition~\ref{prop: q>d} is applicable (i.e., $q \gg d$), the constants appearing in
the results above can be improved accordingly. For the sake of simplicity and
readability, we have chosen not to state these sharpened versions explicitly.
\end{remark}

\subsection{The spaces of degree/level at most $d$}

The previous results extend immediately to the spaces
$\mathcal T_{\le d}(C_q^N)$, $\mathcal P_{\le d}(C_q^N)$ and $\mathcal B_{\le d}(C_q^N)$.

\begin{corollary}
Let $\boldsymbol{\eta}$ denote any of the invariants
$\sid$, $\boldsymbol{\chi}$, or $\boldsymbol{g\!l}$.
Then, for every $q\ge2$ and $1\le d\le N$,
\[
\boldsymbol{\eta}\bigl(\mathcal P_{\le d}(C_q^N)\bigr),
\quad
\boldsymbol{\eta}\bigl(\mathcal T_{\le d}(C_q^N)\bigr),
\quad
\boldsymbol{\eta}\bigl(\mathcal B_{\le d}(C_q^N)\bigr)
\sim_{c(q)^d}
\Bigl(\frac Nd\Bigr)^{\frac{d-1}{2}},
\]
where the constants depend only on $q$.
\end{corollary}

\begin{proof}
We prove the result for $\mathcal B_{\le d}(C_q^N)$; the other cases are
analogous.

For the lower bound, Remark~\ref{rem:level-projection-discrete} gives a
projection $\Pi_d:\mathcal B_{\le d}(C_q^N)\to\mathcal B_d(C_q^N)$ with
$\|\Pi_d\|\le c^d$. Hence, by the ideal property of the Gordon--Lewis
constant,
\[
\boldsymbol{g\!l}\bigl(\mathcal B_d(C_q^N)\bigr)
\le
c^d\,
\boldsymbol{g\!l}\bigl(\mathcal B_{\le d}(C_q^N)\bigr).
\]
The lower bound for $\boldsymbol{g\!l}\bigl(\mathcal B_{\le d}(C_q^N)\bigr)$
now follows from Theorem~\ref{thm:invariants}. Since
$\boldsymbol{g\!l}\le\boldsymbol{\chi}\le\sid$, the same lower bound holds
for all three invariants.

For the upper bound, take $f\in\mathcal B_{\le d}(C_q^N)$ and write
$f=\sum_{k=0}^d f_k$, with $f_k\in\mathcal B_k(C_q^N)$. By
Remark~\ref{rem:level-projection-discrete}, $\|f_k\|_\infty\le c^d\|f\|_\infty$
for every $0\le k\le d$. Therefore, using the Sidon estimates for the exact
levels,
\[
\sum_{\operatorname{lev}(\bm m)\le d}|\widehat f(\bm m)|
\le
c^d
\sum_{k=0}^d
\sid\bigl(\mathcal B_k(C_q^N)\bigr)
\|f\|_\infty
\le
c(q)^d
\Bigl(\frac Nd\Bigr)^{\frac{d-1}{2}}
\|f\|_\infty.
\]
Thus
$\sid(\mathcal B_{\le d}(C_q^N))
\le
c(q)^d (N/d)^{(d-1)/2}$.
Since $\boldsymbol{g\!l}\le\boldsymbol{\chi}\le\sid$, this gives the upper
bound for all three invariants.

The proofs for $\mathcal P_{\le d}(C_q^N)$ and
$\mathcal T_{\le d}(C_q^N)$ are identical, using the corresponding projections
onto the exact homogeneous or tetrahedral levels and
Theorem~\ref{thm:invariants}.
\end{proof}

\section{The projection constant of spherical polynomial spaces}
\label{sec2}

We now study the asymptotic behaviour of the projection constants of the
level--$d$ spherical spaces. The main result of this section identifies a
universal limit governed by Hermite polynomials.

\begin{theorem}\label{thm: main spherical}
For every $q\ge 2$ and every integer $d\ge 0$,
\[
\lim_{N\to\infty}
\frac{\boldsymbol{\lambda}\big(\mathcal B_d(C_q^N)\big)}
     {\sqrt{\dim \mathcal B_d(C_q^N)}}
=
\frac{\mathbb E\big|\mathrm{He}_d(Z)\big|}{\sqrt{d!}},
\]
where $Z$ is a standard real Gaussian random variable.
\end{theorem}

The proof is based on an integral representation of projection constants for
subspaces generated by characters. More precisely, if
$\Lambda\subset \mathbb Z_q^N$, then
\begin{equation}\label{hamming-int}
\boldsymbol{\Lambda}\big(\mathcal P_{\Lambda}(C_q^N)\big)
=
\mathbb E\bigg[
\bigg|
\sum_{\bm m\in\Lambda}\chi_{\bm m}
\bigg|
\bigg],
\end{equation}
where $\mathcal P_{\Lambda}(C_q^N)$ is endowed with the supremum norm inherited
from $C(C_q^N)$.

This identity is a particular case of a general theorem for compact Abelian
groups proved in \cite[Theorem~2.1]{defant2024projection}, building on
classical ideas of Rudin~\cite{rudin1962projections}; see also
\cite[Theorem~III.B.13]{wojtaszczyk1996banach}. We shall denote by
\[
K_{\mathcal P_{\Lambda}(C_q^N)}
:=
\sum_{\bm m\in\Lambda}\chi_{\bm m}
\]
the associated kernel.

In the spherical case the kernel
$K_{\mathcal B_d(C_q^N)}$ is radial, that is, it depends on a point
$x\in C_q^N$ only through its Hamming weight. This allows us to identify it
with a Krawtchouk polynomial and thereby study the projection constant via
the asymptotic theory of these polynomials. The key step is therefore to
obtain an explicit expression for the kernel in terms of
\[
w(x):=\#\{k:x_k\neq 1\}.
\]

\begin{proposition}\label{prop:spherical-kernel-krawtchouk}
Let $x\in C_q^N$ and let $d\ge1$. Then
\[
K_{\mathcal B_d(C_q^N)}(x)
=
\sum_{j=0}^{\min(d,w(x))}
(-1)^j(q-1)^{d-j}
\binom{w(x)}{j}
\binom{N-w(x)}{d-j}.
\]
\end{proposition}

\begin{proof}
By definition,
\[
K_{\mathcal B_d(C_q^N)}(x)
=
\sum_{\substack{\bm m\in\mathbb Z_q^N\\
|\operatorname{supp}(\bm m)|=d}}
\chi_{\bm m}(x)
=
\sum_{\substack{\bm m\in\mathbb Z_q^N\\
|\operatorname{supp}(\bm m)|=d}}
\prod_{k\colon m_k\neq0}x_k^{m_k}.
\]
Grouping the sum according to the support
$S=\operatorname{supp}(\bm m)\subset\{1,\ldots,N\}$, with $|S|=d$, we get
\[
K_{\mathcal B_d(C_q^N)}(x)
=
\sum_{|S|=d}
\prod_{k\in S}
\bigg(\sum_{j=1}^{q-1}x_k^j\bigg).
\]
For each coordinate $k$,
\[
\sum_{j=1}^{q-1}x_k^j
=
\begin{cases}
q-1, & \text{if } x_k=1,\\[4pt]
-1,  & \text{if } x_k\neq 1,
\end{cases}
\]
since $\sum_{j=0}^{q-1}x_k^j=0$ when $x_k\neq1$, whereas it equals $q$ when
$x_k=1$.

Set
\[
u_k:=\sum_{j=1}^{q-1}x_k^j.
\]
There are $N-w(x)$ coordinates for which $u_k=q-1$ and $w(x)$ coordinates for
which $u_k=-1$. Thus
\[
K_{\mathcal B_d(C_q^N)}(x)
=
\sum_{|S|=d}\prod_{k\in S}u_k,
\]
which is the $d$-th elementary symmetric polynomial in the numbers $u_k$.
Choosing $d-j$ factors equal to $q-1$ and $j$ factors equal to $-1$ yields the
formula.
\end{proof}

In particular, the formula shows that $K_{\mathcal B_d(C_q^N)}(x)$ depends on
$x$ only through its Hamming weight $w(x)$ and is given by an explicit
polynomial expression in $w(x)$. This observation motivates the introduction
of the classical Krawtchouk polynomials associated with the Hamming scheme.

\subsection{Krawtchouk polynomials}

Fix  $q\ge2$ and $N\in\mathbb N$, and let
$d,h\in\{0,1,\ldots,N\}$. The \emph{$q$-ary Krawtchouk polynomials} form one
of the classical families of discrete orthogonal polynomials naturally
associated with the Hamming scheme. They are defined by
\[
K_d^{(q)}(h;N)
=
\sum_{j=0}^{\min(d,h)}
(-1)^j(q-1)^{d-j}
\binom{h}{j}
\binom{N-h}{d-j},
\quad d\ge1,
\]
and by $K_0^{(q)}(h;N)\equiv1$. In particular,
\[
K_1^{(q)}(h;N)=(q-1)N-qh,
\quad
K_d^{(q)}(0;N)=(q-1)^d\binom{N}{d},
\quad
K_d^{(q)}(N;N)=(-1)^d\binom{N}{d}.
\]

These polynomials depend only on the parameters $(q,N,d,h)$ and play a
fundamental role in the harmonic analysis of the Hamming scheme, as well as
in the theory of error-correcting codes and association schemes. Standard
references for their properties and normalizations include
\cite[\S9.11]{KoekoekLeskySwarttouw2010}, \cite[Ch.~5]{Szego1975}, and
\cite[\S18.19--\S18.23]{DLMF}, while their combinatorial and algebraic
significance is developed in \cite[Ch.~2]{MacWilliamsSloane} and
\cite[Chs.~II--III]{BannaiIto}.

Comparing the definition above with the explicit expression obtained for the
kernel of the level--$d$ spherical space, we obtain
\[
K_{\mathcal B_d(C_q^N)}(x)
=
K_d^{(q)}\big(w(x);N\big).
\]
This identification is the standard passage from zonal spherical functions on
the Hamming scheme to Krawtchouk polynomials, and it will be used repeatedly
below; see \cite[Ch.~2]{MacWilliamsSloane} and \cite[Ch.~III]{BannaiIto}.

\subsection{Generating function}

The Krawtchouk family admits a simple and explicit generating function, which
provides a compact way to encode the sequence
$\{K_d^{(q)}(h;N)\}_{d=0}^N$ and plays a central role in asymptotic analysis.
For classical references, see for instance
\cite[Eq.~(9.11.1)]{KoekoekLeskySwarttouw2010} or
\cite[DLMF~(18.23.13)]{DLMF}. For completeness, we include a brief proof.

\begin{proposition}\label{prop:generating-function}
For every $q\ge2$, $N\ge1$, and $h\in\{0,1,\ldots,N\}$,
\[
\sum_{d=0}^{N} K_d^{(q)}(h;N)z^d
=
\big(1+(q-1)z\big)^{N-h}(1-z)^h.
\]
\end{proposition}

\begin{proof}
Let $z$ be a formal variable. Using the explicit formula for
$K_d^{(q)}(h;N)$, we compute
\begin{align*}
\sum_{d=0}^{N} K_d^{(q)}(h;N)z^d
&=
\sum_{d=0}^{N}
\sum_{j=0}^{\min(d,h)}
(-1)^j(q-1)^{d-j}
\binom{h}{j}\binom{N-h}{d-j}z^d \\
&=
\sum_{j=0}^{h}
(-1)^j\binom{h}{j}z^j
\sum_{d=0}^{N}
(q-1)^{d-j}\binom{N-h}{d-j}z^{d-j}.
\end{align*}
In the inner sum, set $a=d-j$. Since $\binom{N-h}{a}=0$ unless
$0\le a\le N-h$, we obtain
\[
\sum_{d=0}^{N}
(q-1)^{d-j}\binom{N-h}{d-j}z^{d-j}
=
\sum_{a=0}^{N-h}
\binom{N-h}{a}\big((q-1)z\big)^a
=
\big(1+(q-1)z\big)^{N-h}.
\]
Therefore,
\[
\sum_{d=0}^{N} K_d^{(q)}(h;N)z^d
=
\bigg(\sum_{j=0}^{h}\binom{h}{j}(-z)^j\bigg)
\big(1+(q-1)z\big)^{N-h}
=
(1-z)^h\big(1+(q-1)z\big)^{N-h},
\]
which completes the proof.
\end{proof}

\subsection{Orthogonality and probabilistic normalization}

Let $X=(X_1,\ldots,X_N)$ be the canonical random vector on $C_q^N$, with
independent coordinates $X_k\sim\mathrm{Uniform}(C_q)$. Define the Hamming
weight
\[
W_N:=w(X)=\#\{k\colon X_k\neq1\}.
\]
Then $W_N$ is a binomial random variable with parameters $(N,p)$, namely
\begin{equation}\label{binomial}
W_N\sim \mathrm{Bin}(N,p),
\quad
p=\frac{q-1}{q}.
\end{equation}

With respect to this distribution, the Krawtchouk polynomials form an
orthogonal family. Equivalently, for all $d,d'\in\{0,\ldots,N\}$,
\[
\mathbb E\big[
K_d^{(q)}(W_N;N)\,K_{d'}^{(q)}(W_N;N)
\big]
=
(q-1)^d\binom{N}{d}\,\delta_{d,d'}.
\]
In particular,
\begin{equation}\label{orthogonality}
\mathbb E\big[
\big|K_d^{(q)}(W_N;N)\big|^2
\big]
=
(q-1)^d\binom{N}{d},
\end{equation}
which coincides with the dimension of the level--$d$ spherical space.

This probabilistic formulation is equivalent to the classical orthogonality
relations for Krawtchouk polynomials with respect to the binomial weight
$\binom{N}{h}(q-1)^h$. See, for instance, \cite[Ch.~5]{Szego1975},
\cite[\S9.11]{KoekoekLeskySwarttouw2010}, and
\cite[\S18.19--\S18.23]{DLMF}.

\medskip
\noindent\textbf{Three-term recurrence.}
For completeness, we record an additional classical property of the
$q$-ary Krawtchouk polynomials, although it will not be needed in the main
arguments.

For $d\ge1$, the Krawtchouk polynomials satisfy the classical three-term
recurrence relation
\[
(d+1)\,K_{d+1}^{(q)}(h;N)
=
\bigl((q-1)N-qh-(q-2)d\bigr)\,K_d^{(q)}(h;N)
-
(q-1)(N-d+1)\,K_{d-1}^{(q)}(h;N),
\]
with initial conditions $K_0^{(q)}\equiv1$ and
$K_1^{(q)}(h;N)=(q-1)N-qh$.
This recurrence allows one to compute the family inductively in the degree
and can be found, for instance, in
\cite[\S9.11]{KoekoekLeskySwarttouw2010} and \cite[Ch.~5]{Szego1975}.

\subsection{From Krawtchouk to Hermite: a probabilistic limit}

We recall that a sequence of random variables $X_N$ converges in distribution
to $X$, written $X_N\overset{D}{\longrightarrow}X$, if
\[
\lim_{N\to\infty}\mathbb P(X_N\le t)=\mathbb P(X\le t)
\]
for every continuity point $t$ of the distribution function of $X$.
Equivalently,
\[
\lim_{N\to\infty}\mathbb E[\varphi(X_N)]
=
\mathbb E[\varphi(X)]
\]
for every bounded continuous function $\varphi\colon\mathbb R\to\mathbb R$.
We also recall that $X_N$ converges to $X$ in probability, written
$X_N\overset{\mathbb P}{\longrightarrow}X$, if, for every $\varepsilon>0$,
\[
\lim_{N\to\infty}\mathbb P(|X_N-X|>\varepsilon)=0.
\]

The following theorem describes the asymptotic behavior of Krawtchouk
polynomials when evaluated at the random Hamming weight.

\begin{theorem}\label{thm:probabilistic-limit}
Let $\{W_N\}_{N\ge1}$ be a sequence of random variables such that
\[
W_N \sim \mathrm{Bin}(N,p),
\qquad
p=\frac{q-1}{q}.
\]
Then, for every fixed $d\ge0$,
\[
\frac{K_d^{(q)}(W_N;N)}{\sqrt{(q-1)^d \binom{N}{d}}}
\;\overset{D}{\longrightarrow}\;
\frac{\mathrm{He}_d(Z)}{\sqrt{d!}},
\qquad N\to\infty,
\]
where $Z$ is a standard Gaussian random variable and
$\mathrm{He}_d$ denotes the $d$-th probabilists' Hermite polynomial.
\end{theorem}

The guiding intuition behind this result is that Krawtchouk polynomials form
an orthogonal basis associated with the binomial distribution, while Hermite
polynomials play the analogous role for the Gaussian law. Consequently, when
the binomial distribution is appropriately rescaled and converges to a
Gaussian limit, the corresponding orthogonal polynomials are expected to
converge to Hermite polynomials. Indeed, by the central limit theorem, the
centered and normalized Hamming weight
\[
Z_N := \frac{W_N-Np}{\sigma\sqrt{N}},
\qquad \sigma^2=p(1-p),
\]
converges in distribution to a standard Gaussian random variable $Z$ as
$N\to\infty$. The probabilistic convergence above is a consequence of a
stronger, deterministic asymptotic statement, which we now formulate.

\begin{proposition}\label{prop:krawtchouk-hermite-deterministic}
Fix $q\ge2$ and $d\ge0$, and define
\[
p:=\frac{q-1}{q},
\qquad
\sigma^2:=p(1-p).
\]
Then, for every compact set $K\subset\mathbb R$,
\[
\frac{K_d^{(q)}\big(\lfloor Np+\sigma\sqrt N\,x\rfloor;N\big)}
     {\sqrt{(q-1)^d\binom{N}{d}}}
\longrightarrow
\frac{(-1)^d}{\sqrt{d!}}\,\mathrm{He}_d(x)
\]
uniformly for $x\in K$, as $N\to\infty$.
\end{proposition}

\begin{proof}
Fix $M,T>0$ and work uniformly for $|x|\le M$. Throughout the proof, all
implicit constants may depend on $M$ and $T$, but are independent of $N$.

For $N$ sufficiently large, define
\[
h_N(x):=\big\lfloor Np+\sigma\sqrt N\,x\big\rfloor\in\{0,\ldots,N\}.
\]
Then
\[
h_N(x)=Np+\sigma\sqrt N\,x+\varepsilon_N(x),
\qquad
|\varepsilon_N(x)|\le 1.
\]
Introduce the complex variable
\[
z:=-\,\frac{t}{\sqrt N\,\sqrt{q-1}},
\qquad |t|=T.
\]
By Proposition~\ref{prop:generating-function}, for every
$h\in\{0,\ldots,N\}$,
\[
\sum_{m=0}^{N} K_m^{(q)}(h;N)\,z^m
=
(1+(q-1)z)^{N-h}(1-z)^h.
\]
We apply this identity with $h=h_N(x)$. Set
\[
\Phi_N(t,x)
:=
(N-h_N(x))\log(1+(q-1)z)+h_N(x)\log(1-z).
\]
Since $|z|=O(N^{-1/2})$ and $|(q-1)z|=O(N^{-1/2})$, Taylor's expansion of the
logarithm around zero, taken in the principal branch, gives
\[
\log(1+w)=w-\frac{w^2}{2}+\rho_3(w),
\qquad
|\rho_3(w)|\le C|w|^3
\]
for $|w|$ sufficiently small. Applying this expansion, we obtain
\[
\begin{aligned}
\Phi_N(t,x)
&=
(N-h_N(x))
\bigg((q-1)z-\frac{(q-1)^2z^2}{2}\bigg)
+h_N(x)\bigg(-z-\frac{z^2}{2}\bigg)
+R_N(t,x),
\end{aligned}
\]
where the remainder satisfies
\begin{equation}\label{eq:krawtchouk-remainder}
|R_N(t,x)|
\le
C\big((N-h_N(x))|(q-1)z|^3+h_N(x)|z|^3\big)
=
O_{M,T}(N^{-1/2}),
\end{equation}
uniformly for $|x|\le M$ and $|t|=T$.

We now isolate the linear and quadratic contributions. The linear part is
\[
L_N(t,x)
:=
\big((N-h_N(x))(q-1)-h_N(x)\big)z
=
\big((q-1)N-qh_N(x)\big)z.
\]
Using
\[
h_N(x)=Np+\sigma\sqrt N\,x+\varepsilon_N(x),
\qquad
p=\frac{q-1}{q},
\]
we obtain
\[
(q-1)N-qh_N(x)
=
-q\sigma\sqrt N\,x-q\varepsilon_N(x).
\]
Since
\[
\sigma^2=p(1-p)=\frac{q-1}{q^2},
\qquad
q\sigma=\sqrt{q-1},
\]
and
\[
z=-\frac{t}{\sqrt N\,\sqrt{q-1}},
\]
it follows that
\[
L_N(t,x)
=
xt+\frac{q\varepsilon_N(x)}{\sqrt N\,\sqrt{q-1}}\,t.
\]
Therefore,
\[
L_N(t,x)=xt+O_{M,T}(N^{-1/2}),
\]
uniformly for $|x|\le M$ and $|t|=T$.

On the other hand, if we look at the quadratic part, which is
\[
Q_N(t,x)
:=
-\frac12\big((N-h_N(x))(q-1)^2+h_N(x)\big)z^2,
\]
since
$
z^2=\frac{t^2}{N(q-1)}$,
we may rewrite this as
\[
Q_N(t,x)
=
-\frac{t^2}{2}
\bigg(
\frac{(N-h_N(x))(q-1)}{N}
+
\frac{h_N(x)}{N(q-1)}
\bigg).
\]
Once again we use the fact that
\[
h_N(x)=Np+O_M(\sqrt N),
\qquad
N-h_N(x)=N(1-p)+O_M(\sqrt N),
\]
uniformly for $|x|\le M$, and hence
\[
\frac{(N-h_N(x))(q-1)}{N}
+
\frac{h_N(x)}{N(q-1)}
=
(1-p)(q-1)+\frac{p}{q-1}+O_M(N^{-1/2}).
\]
In view of $p=(q-1)/q$, we compute
\[
(1-p)(q-1)+\frac{p}{q-1}
=
\frac{q-1}{q}+\frac1q
=
1.
\]
Therefore,
\[
Q_N(t,x)
=
-\frac{t^2}{2}\bigl(1+O_{M,T}(N^{-1/2})\bigr),
\]
uniformly for $|x|\le M$ and $|t|=T$. Combining the linear term, the quadratic
term, and the remainder estimate from \eqref{eq:krawtchouk-remainder}, we
conclude that
\[
\Phi_N(t,x)
=
xt-\frac{t^2}{2}+O_{M,T}(N^{-1/2}),
\]
uniformly for $|x|\le M$ and $|t|=T$.

Exponentiating, we obtain
\[
(1+(q-1)z)^{N-h_N(x)}(1-z)^{h_N(x)}
=
e^{\Phi_N(t,x)}
=
e^{xt-t^2/2}\bigl(1+r_N(t,x)\bigr),
\]
where
\[
\sup_{\substack{|x|\le M\\ |t|=T}} |r_N(t,x)|
=
O_{M,T}(N^{-1/2}).
\]
Using the generating function identity with $h=h_N(x)$, and recalling that
\[
z=-\frac{t}{\sqrt N\,\sqrt{q-1}},
\]
we may write
\[
\sum_{m=0}^N K_m^{(q)}(h_N(x);N)\,z^m
=
\sum_{m=0}^N
\frac{K_m^{(q)}(h_N(x);N)}
     {(\sqrt{q-1})^m N^{m/2}}\,(-t)^m
=:F_N(t,x).
\]
Hence
\begin{equation}\label{eq:krawtchouk-FN}
F_N(t,x)
=
e^{xt-t^2/2}\bigl(1+r_N(t,x)\bigr),
\qquad |x|\le M,\quad |t|=T.
\end{equation}

Fix $|x|\le M$. Since
\[
F_N(t,x)
=
\sum_{m=0}^N
\frac{K_m^{(q)}(h_N(x);N)}
     {(\sqrt{q-1})^m N^{m/2}}\,(-t)^m,
\]
Cauchy's integral formula on the circle $|t|=T$ gives
\[
\frac{(-1)^d}{(\sqrt{q-1})^d N^{d/2}}\,
K_d^{(q)}(h_N(x);N)
=
\frac{1}{2\pi i}
\oint_{|t|=T}\frac{F_N(t,x)}{t^{d+1}}\,dt.
\]
Using \eqref{eq:krawtchouk-FN}, we split the right-hand side as
\[
\frac{1}{2\pi i}\oint_{|t|=T}
\frac{e^{xt-t^2/2}}{t^{d+1}}\,dt
+
\frac{1}{2\pi i}\oint_{|t|=T}
\frac{e^{xt-t^2/2}r_N(t,x)}{t^{d+1}}\,dt.
\]
The generating function of the probabilists' Hermite polynomials,
\[
e^{xt-t^2/2}
=
\sum_{m=0}^\infty \frac{\mathrm{He}_m(x)}{m!}\,t^m,
\]
implies
\[
\frac{1}{2\pi i}\oint_{|t|=T}
\frac{e^{xt-t^2/2}}{t^{d+1}}\,dt
=
\frac{\mathrm{He}_d(x)}{d!}.
\]
For the error term, since $|x|\le M$ and $|t|=T$,
\[
\sup_{\substack{|x|\le M\\ |t|=T}}
|e^{xt-t^2/2}|
\le C_{M,T},
\]
and therefore
\[
\left|
\frac{1}{2\pi i}\oint_{|t|=T}
\frac{e^{xt-t^2/2}r_N(t,x)}{t^{d+1}}\,dt
\right|
\le
C_{M,T}\sup_{\substack{|x|\le M\\ |t|=T}}|r_N(t,x)|
=
O_{M,T}(N^{-1/2}).
\]
Consequently,
\[
\frac{(-1)^d}{(\sqrt{q-1})^d N^{d/2}}\,
K_d^{(q)}(h_N(x);N)
=
\frac{\mathrm{He}_d(x)}{d!}+O_{M,T}(N^{-1/2}),
\]
uniformly for $|x|\le M$.

For fixed $d$, we have
\[
\binom{N}{d}
=
\frac{N^d}{d!}\bigl(1+O_d(N^{-1})\bigr),
\]
and therefore
\[
\sqrt{(q-1)^d\binom{N}{d}}
=
\frac{(\sqrt{q-1})^d N^{d/2}}{\sqrt{d!}}
\bigl(1+O_d(N^{-1})\bigr).
\]
Combining this with the estimate above, we get
\[
\frac{K_d^{(q)}(h_N(x);N)}
     {\sqrt{(q-1)^d\binom{N}{d}}}
=
\frac{(-1)^d}{\sqrt{d!}}\,\mathrm{He}_d(x)
+
O_{M,T,d}(N^{-1/2}),
\]
uniformly for $|x|\le M$. Since $M>0$ was arbitrary, this proves uniform
convergence on compact subsets of $\mathbb R$.
\end{proof}

We now turn to the proof of Theorem~\ref{thm:probabilistic-limit}, explaining
how the probabilistic convergence follows from the deterministic asymptotic
result in Proposition~\ref{prop:krawtchouk-hermite-deterministic}. We shall
use the following standard consequence of Slutsky's theorem: if
$X_N\overset{D}{\longrightarrow}X$ and
$Y_N\overset{\mathbb P}{\longrightarrow}0$, then
$X_N+Y_N\overset{D}{\longrightarrow}X$.

\begin{proof}[Proof of Theorem~\ref{thm:probabilistic-limit}]
Set
\[
Z_N:=\frac{W_N-Np}{\sigma\sqrt N},
\qquad
\sigma^2=p(1-p).
\]
By the central limit theorem,
\[
Z_N\overset{D}{\longrightarrow}Z\sim\mathcal N(0,1).
\]
In particular, the family $\{Z_N\}_N$ is tight.

Define
\[
f_N(x)
:=
\frac{K_d^{(q)}\big(\lfloor Np+\sigma\sqrt N\,x\rfloor;N\big)}
     {\sqrt{(q-1)^d\binom{N}{d}}},
\qquad
f(x):=\frac{(-1)^d}{\sqrt{d!}}\,\mathrm{He}_d(x).
\]
By Proposition~\ref{prop:krawtchouk-hermite-deterministic}, $f_N\to f$
uniformly on compact subsets of $\mathbb R$, and $f$ is continuous. Moreover,
since
\[
Np+\sigma\sqrt N\,Z_N=W_N\in\mathbb Z,
\]
we have
\[
\lfloor Np+\sigma\sqrt N\,Z_N\rfloor=W_N.
\]
Therefore
\begin{equation}\label{eq:fnZn}
f_N(Z_N)
=
\frac{K_d^{(q)}(W_N;N)}
     {\sqrt{(q-1)^d\binom{N}{d}}}.
\end{equation}

We claim that
\[
f_N(Z_N)\overset{D}{\longrightarrow} f(Z).
\]
Indeed, since $\{Z_N\}_N$ is tight, for every $\varepsilon>0$ there exists
$M<\infty$ such that
\[
\sup_N\mathbb P(|Z_N|>M)\le\varepsilon.
\]
By the locally uniform convergence $f_N\to f$, there exists $N_0$ such that,
for all $N\ge N_0$,
\[
\sup_{|x|\le M}|f_N(x)-f(x)|\le\varepsilon.
\]
Hence, for $N\ge N_0$,
\[
\mathbb P\big(|f_N(Z_N)-f(Z_N)|>\varepsilon\big)
\le
\mathbb P(|Z_N|>M)
\le
\varepsilon.
\]
Thus
\[
f_N(Z_N)-f(Z_N)\overset{\mathbb P}{\longrightarrow}0.
\]
On the other hand, since $f$ is continuous and
$Z_N\overset{D}{\longrightarrow}Z$, the continuous mapping theorem gives
\[
f(Z_N)\overset{D}{\longrightarrow}f(Z).
\]
Slutsky's theorem therefore yields
\[
f_N(Z_N)\overset{D}{\longrightarrow}f(Z).
\]
Combining this with \eqref{eq:fnZn}, we get
\[
\frac{K_d^{(q)}(W_N;N)}
     {\sqrt{(q-1)^d\binom{N}{d}}}
\overset{D}{\longrightarrow}
\frac{(-1)^d}{\sqrt{d!}}\,\mathrm{He}_d(Z).
\]
Finally, since $Z$ is symmetric and
\[
\mathrm{He}_d(-x)=(-1)^d\mathrm{He}_d(x),
\]
the random variables $(-1)^d\mathrm{He}_d(Z)$ and $\mathrm{He}_d(Z)$ have the
same distribution. This gives the stated form of the limit.
\end{proof}

\subsection{Proof of Theorem~\ref{thm: main spherical}}

Fix $q\ge2$ and $d\ge0$. Recall from \eqref{hamming-int} that the
projection constant of the level--$d$ spherical space is given by
\[
\boldsymbol{\lambda}\big(\mathcal B_d(C_q^N)\big)
=
\|K_{\mathcal B_d(C_q^N)}\|_{L_1(C_q^N)}
=
\mathbb E\big|K_{\mathcal B_d(C_q^N)}(X)\big|,
\]
where, for each $N\in\mathbb N$, the random vector
$X=(X_1,\ldots,X_N)$ is uniformly distributed on $C_q^N$. Setting
$W_N:=w(X)$, we have
\[
W_N\sim\mathrm{Bin}(N,p),
\qquad
p=\frac{q-1}{q}.
\]
Since
\[
K_{\mathcal B_d(C_q^N)}(x)=K_d^{(q)}(w(x);N),
\]
we can rewrite the projection constant as
\[
\boldsymbol{\lambda}\big(\mathcal B_d(C_q^N)\big)
=
\mathbb E\big|K_d^{(q)}(W_N;N)\big|.
\]

For $N\ge d$, define the normalized random variables
\[
Y_N
:=
\frac{K_d^{(q)}(W_N;N)}
     {\sqrt{(q-1)^d\binom{N}{d}}}.
\]
Since
\[
\dim\mathcal B_d(C_q^N)=(q-1)^d\binom{N}{d},
\]
we obtain
\[
\frac{\boldsymbol{\lambda}\big(\mathcal B_d(C_q^N)\big)}
     {\sqrt{\dim\mathcal B_d(C_q^N)}}
=
\mathbb E|Y_N|.
\]

By the orthogonality relation for Krawtchouk polynomials with respect to the
binomial weight, see \eqref{orthogonality},
\[
\mathbb E\big|K_d^{(q)}(W_N;N)\big|^2
=
(q-1)^d\binom{N}{d}.
\]
Therefore
\[
\mathbb E|Y_N|^2=1,
\qquad N\ge d.
\]
In particular, the family $\{Y_N\}_{N\ge d}$ is uniformly integrable. Indeed,
for every $R>0$, by Cauchy--Schwarz and Markov's inequality,
\[
\sup_{N\ge d}
\mathbb E\big(|Y_N|\mathbf 1_{\{|Y_N|>R\}}\big)
\le
\sup_{N\ge d}\big(\mathbb E|Y_N|^2\big)^{1/2}
\sup_{N\ge d}\mathbb P(|Y_N|>R)^{1/2}
\le
\frac{1}{R}.
\]

On the other hand, Theorem~\ref{thm:probabilistic-limit} gives
\[
Y_N\overset{D}{\longrightarrow}
Y
:=
\frac{\mathrm{He}_d(Z)}{\sqrt{d!}},
\qquad N\to\infty,
\]
where $Z\sim\mathcal N(0,1)$.

We now show that $\mathbb E|Y_N|\to\mathbb E|Y|$. Let $\varepsilon>0$ be
arbitrary. By uniform integrability of $\{Y_N\}_{N\ge d}$ and integrability of
$Y$, choose $R>0$ such that
\[
\sup_{N\ge d}
\mathbb E\big(|Y_N|\mathbf 1_{\{|Y_N|>R\}}\big)<\varepsilon
\quad\text{and}\quad
\mathbb E\big(|Y|\mathbf 1_{\{|Y|>R\}}\big)<\varepsilon.
\]
For this fixed $R$, consider the truncation
\[
g_R(t):=\min\{|t|,R\}.
\]
Since $g_R$ is bounded and continuous, convergence in distribution implies
\[
\mathbb E\,g_R(Y_N)\longrightarrow \mathbb E\,g_R(Y).
\]
Moreover,
\[
0\le
\mathbb E|Y_N|-\mathbb E\,g_R(Y_N)
\le
\mathbb E\big(|Y_N|\mathbf 1_{\{|Y_N|>R\}}\big),
\]
and similarly
\[
0\le
\mathbb E|Y|-\mathbb E\,g_R(Y)
=
\mathbb E\big(|Y|\mathbf 1_{\{|Y|>R\}}\big).
\]
Combining these estimates, we obtain
\[
\limsup_{N\to\infty}
\big|\mathbb E|Y_N|-\mathbb E|Y|\big|
\le
2\varepsilon.
\]
Since $\varepsilon>0$ is arbitrary, this proves that
\[
\mathbb E|Y_N|\longrightarrow \mathbb E|Y|
=
\mathbb E\bigg|\frac{\mathrm{He}_d(Z)}{\sqrt{d!}}\bigg|.
\]
Therefore,
\[
\lim_{N\to\infty}
\frac{\boldsymbol{\lambda}\big(\mathcal B_d(C_q^N)\big)}
     {\sqrt{\dim\mathcal B_d(C_q^N)}}
=
\frac{\mathbb E|\mathrm{He}_d(Z)|}{\sqrt{d!}},
\]
which completes the proof of Theorem~\ref{thm: main spherical}.

\subsection{The top-layer principle for spherical spaces}

Finally, we study the projection constant of the space
$\mathcal B_{\le d}(C_q^N)$. The guiding principle is that, after
normalization by the square root of the dimension, the contribution of the
lower levels or degrees is asymptotically negligible. Thus the limiting
behaviour is determined by the top level or top degree.

\begin{proposition}\label{prop:B-le-d}
Fix $q\ge2$ and $d\ge1$. Then
\[
\lim_{N\to\infty}
\frac{\boldsymbol{\lambda}\big(\mathcal B_{\le d}(C_q^N)\big)}
     {\sqrt{\dim \mathcal B_{\le d}(C_q^N)}}
=
\frac{1}{\sqrt{d!}}\,\mathbb E\big|\mathrm{He}_d(Z)\big|,
\qquad Z\sim\mathcal N(0,1).
\]
\end{proposition}

\begin{proof}
Let $X=(X_1,\ldots,X_N)$ be uniformly distributed on $C_q^N$. By
\eqref{hamming-int}, applied to the index set
\[
\mathcal S_{\le d}
:=
\big\{\bm m\in\{0,\ldots,q-1\}^N\colon
|\operatorname{supp}(\bm m)|\le d
\big\},
\]
the projection constant admits the integral formula
\[
\boldsymbol{\lambda}\big(\mathcal B_{\le d}(C_q^N)\big)
=
\mathbb E\big|K_{\mathcal B_{\le d}(C_q^N)}(X)\big|,
\]
where
\[
K_{\mathcal B_{\le d}(C_q^N)}(x)
:=
\sum_{\bm m\in\mathcal S_{\le d}}\chi_{\bm m}(x).
\]
Decomposing according to the support size, we write
\[
K_{\mathcal B_{\le d}(C_q^N)}(x)
=
\sum_{k=0}^d K_{\mathcal B_k(C_q^N)}(x),
\]
where
\[
K_{\mathcal B_k(C_q^N)}(x)
:=
\sum_{\substack{\bm m\in\{0,\ldots,q-1\}^N\\
|\operatorname{supp}(\bm m)|=k}}
\chi_{\bm m}(x).
\]
Since
\[
\dim \mathcal B_k(C_q^N)
=
(q-1)^k\binom{N}{k},
\]
we have
\[
\dim \mathcal B_{\le d}(C_q^N)
=
\sum_{k=0}^d (q-1)^k\binom{N}{k}
\sim
(q-1)^d\frac{N^d}{d!},
\qquad N\to\infty.
\]
In particular,
\[
\frac{\dim \mathcal B_d(C_q^N)}
     {\dim \mathcal B_{\le d}(C_q^N)}
\longrightarrow 1.
\]

Define the normalized random variables
\[
Y_N
:=
\frac{K_{\mathcal B_{\le d}(C_q^N)}(X)}
     {\sqrt{\dim \mathcal B_{\le d}(C_q^N)}},
\qquad
Y_{N,d}
:=
\frac{K_{\mathcal B_d(C_q^N)}(X)}
     {\sqrt{\dim \mathcal B_d(C_q^N)}}.
\]
Using the dimension asymptotics above, we may write
\[
Y_N
=
Y_{N,d}
\sqrt{\frac{\dim \mathcal B_d(C_q^N)}
           {\dim \mathcal B_{\le d}(C_q^N)}}
+
\sum_{k=0}^{d-1}
\frac{K_{\mathcal B_k(C_q^N)}(X)}
     {\sqrt{\dim \mathcal B_{\le d}(C_q^N)}}.
\]
For $k<d$, orthogonality gives
\[
\mathbb E\big|K_{\mathcal B_k(C_q^N)}(X)\big|^2
=
\dim \mathcal B_k(C_q^N)
=
O_{d,q}(N^k).
\]
Therefore
\[
\frac{K_{\mathcal B_k(C_q^N)}(X)}
     {\sqrt{\dim \mathcal B_{\le d}(C_q^N)}}
=
O_{\mathbb P}\big(N^{(k-d)/2}\big)
\longrightarrow 0
\]
in probability. Since also
\[
\frac{\dim \mathcal B_d(C_q^N)}
     {\dim \mathcal B_{\le d}(C_q^N)}
\longrightarrow 1,
\]
we conclude that
\[
Y_N=Y_{N,d}+o_{\mathbb P}(1).
\]

By Theorem~\ref{thm:probabilistic-limit},
\[
Y_{N,d}
\overset{D}{\longrightarrow}
\frac{\mathrm{He}_d(Z)}{\sqrt{d!}},
\qquad Z\sim\mathcal N(0,1).
\]
Moreover, by orthogonality,
\[
\mathbb E|Y_{N,d}|^2=1.
\]
Also,
\[
\mathbb E|Y_N|^2=1,
\]
because the characters appearing in $K_{\mathcal B_{\le d}(C_q^N)}$ are
orthonormal and the normalization is by
$\dim\mathcal B_{\le d}(C_q^N)$. Hence $\{Y_N\}_N$ is uniformly integrable.
Since $Y_N-Y_{N,d}\to0$ in probability, Slutsky's theorem gives
\[
Y_N
\overset{D}{\longrightarrow}
\frac{\mathrm{He}_d(Z)}{\sqrt{d!}}.
\]
Uniform integrability then yields
\[
\mathbb E|Y_N|
\longrightarrow
\mathbb E\bigg|\frac{\mathrm{He}_d(Z)}{\sqrt{d!}}\bigg|
=
\frac{1}{\sqrt{d!}}\mathbb E|\mathrm{He}_d(Z)|.
\]
Finally,
\[
\frac{\boldsymbol{\lambda}\big(\mathcal B_{\le d}(C_q^N)\big)}
     {\sqrt{\dim \mathcal B_{\le d}(C_q^N)}}
=
\mathbb E|Y_N|,
\]
and the claimed asymptotic formula follows.
\end{proof}

\section{The projection constant of total-degree polynomial spaces}

In contrast with the spherical setting, the asymptotic behaviour of the
homogeneous and tetrahedral spaces exhibits a dichotomy between the Boolean
and non-Boolean regimes. In the Boolean case, the limiting constants are
governed by Hermite polynomials, whereas for $q\ge3$ they are
determined by moments of a circular complex Gaussian.

\begin{theorem}\label{thm:asymptotic-projection-Pd}
For every $q\ge2$ and every integer $d\ge1$, both the homogeneous and the
tetrahedral spaces satisfy
\[
\lim_{N\to\infty}
\frac{\boldsymbol{\lambda}\big(\mathcal P_d(C_q^N)\big)}
     {\sqrt{\dim \mathcal P_d(C_q^N)}}
=
\lim_{N\to\infty}
\frac{\boldsymbol{\lambda}\big(\mathcal T_d(C_q^N)\big)}
     {\sqrt{\dim \mathcal T_d(C_q^N)}}
=
\begin{cases}
\displaystyle
\frac{1}{\sqrt{d!}}\,\mathbb E\big|\mathrm{He}_d(Z)\big|,
 & \text{if } q=2,\\[3ex]
\displaystyle
\frac{\Gamma\big(1+\frac{d}{2}\big)}{\sqrt{d!}},
 & \text{if } q\ge3,
\end{cases}
\]
where $Z$ is a standard real Gaussian random variable.
\end{theorem}

We begin with the homogeneous case
and derive a convenient generating function for the corresponding projection
kernel.

For $d\ge0$, define
\[
K_{\mathcal P_d(C_q^N)}(x)
:=
\sum_{\substack{\bm\alpha\in\{0,\ldots,q-1\}^N\\ |\bm\alpha|=d}}
\chi_{\bm\alpha}(x),
\qquad x\in C_q^N.
\]
By \eqref{hamming-int},
\[
\boldsymbol{\lambda}\big(\mathcal P_d(C_q^N)\big)
=
\mathbb E\big|K_{\mathcal P_d(C_q^N)}(X)\big|,
\]
where $X$ is uniformly distributed on $C_q^N$.

\begin{proposition}\label{prop:dim-poly-alternating}
Fix $q\ge2$ and $d\ge1$. Then
\[
\dim \mathcal P_d(C_q^N)
=
\#\big\{\bm\alpha\in\{0,1,\ldots,q-1\}^N\colon |\bm\alpha|=d\big\}
=
\sum_{j=0}^{\lfloor d/q\rfloor}
(-1)^j\binom{N}{j}\binom{N+d-qj-1}{N-1}.
\]
\end{proposition}

\begin{proof}
By definition,
\[
\dim \mathcal P_d(C_q^N)
=
\#\big\{\bm\alpha=(\alpha_1,\ldots,\alpha_N)\in\{0,1,\ldots,q-1\}^N
\colon \alpha_1+\cdots+\alpha_N=d\big\}.
\]
Let $A$ denote the set of all $\bm\alpha\in\mathbb N_0^N$ such that
$\alpha_1+\cdots+\alpha_N=d$, with no upper bound on the coordinates. For
each $i\in\{1,\ldots,N\}$, let
\[
A_i:=\{\bm\alpha\in A\colon \alpha_i\ge q\}.
\]
The desired set is precisely $A\setminus\bigcup_{i=1}^N A_i$. Hence, by
inclusion--exclusion,
\begin{equation}\label{eq:IE}
\#\bigg(A\setminus\bigcup_{i=1}^N A_i\bigg)
=
\sum_{J\subset\{1,\ldots,N\}}(-1)^{|J|}
\#\bigg(\bigcap_{i\in J}A_i\bigg).
\end{equation}
For $J=\emptyset$ then $\bigcap_{i\in J}A_i$ is by convention equal to $A$.
We now compute $\#(\cap_{i\in J}A_i)$ for a fixed subset $J$ with $|J|=j$.
The condition $\alpha_i\ge q$ for every $i\in J$ is equivalent, after the
change of variables
\[
\beta_i=
\begin{cases}
\alpha_i-q, & i\in J,\\
\alpha_i, & i\notin J,
\end{cases}
\]
to requiring $\beta_i\in\mathbb N_0$ for all $i$ and
\[
\sum_{i=1}^N\beta_i=d-qj.
\]
Therefore $\#(\cap_{i\in J}A_i)$ equals the number of weak compositions of
$d-qj$ into $N$ parts, namely
\[
\binom{N+d-qj-1}{N-1},
\]
with the convention that this number is $0$ whenever $d-qj<0$. Substituting
this into \eqref{eq:IE} and grouping subsets $J$ by their size $j$, we obtain
\[
\dim \mathcal P_d(C_q^N)
=
\sum_{j\ge0}
(-1)^j\binom{N}{j}\binom{N+d-qj-1}{N-1}.
\]
The sum truncates at $j\le\lfloor d/q\rfloor$, since otherwise $d-qj<0$.
This proves the formula.
\end{proof}

As an immediate consequence of Proposition~\ref{prop:dim-poly-alternating}, we
record the standard fixed-degree asymptotics.

\begin{corollary}\label{dimoo}
Fix $d\ge1$ and $q\ge2$. Then
\[
\dim \mathcal P_d(C_q^N)
=
\frac{N^d}{d!}\bigg(1+O\bigg(\frac1N\bigg)\bigg)
\sim
\frac{N^d}{d!},
\qquad N\to\infty.
\]
\end{corollary}

\begin{proof}
By Proposition~\ref{prop:dim-poly-alternating}, we may write
\[
\dim \mathcal P_d(C_q^N)
=
\binom{N+d-1}{N-1}
+
\sum_{j=1}^{\lfloor d/q\rfloor}
(-1)^j
\binom{N}{j}\binom{N+d-qj-1}{N-1}.
\]
The leading term satisfies
\[
\binom{N+d-1}{N-1}
=
\binom{N+d-1}{d}
=
\frac{N^d}{d!}+O(N^{d-1}).
\]
For each fixed $j\ge1$,
\[
\binom{N}{j}\binom{N+d-qj-1}{N-1}
=
\bigg(\frac{N^j}{j!}+O(N^{j-1})\bigg)
\bigg(\frac{N^{d-qj}}{(d-qj)!}+O(N^{d-qj-1})\bigg)
=
O\big(N^{d-(q-1)j}\big).
\]
Since $q\ge2$ and $j\ge1$, we have $d-(q-1)j\le d-1$, and the sum contains
only finitely many terms. Thus all correction terms contribute at most
$O(N^{d-1})$. Therefore,
\[
\dim \mathcal P_d(C_q^N)
=
\frac{N^d}{d!}+O(N^{d-1}),
\]
which yields the stated asymptotics.
\end{proof}

\begin{proposition}\label{prop:GF-slice}
For every $x=(x_1,\ldots,x_N)\in C_q^N$ and every $z\in\C$ with $|z|<1$,
\[
\sum_{d\ge0} K_{\mathcal P_d(C_q^N)}(x)\,z^d
=
\prod_{k=1}^N\bigg(\sum_{a=0}^{q-1}(x_k z)^a\bigg)
=
(1-z^q)^N\prod_{k=1}^N(1-x_k z)^{-1}.
\]
\end{proposition}

\begin{proof}
By definition,
\[
K_{\mathcal P_d(C_q^N)}(x)
=
\sum_{\substack{\alpha\in\{0,\ldots,q-1\}^N\\ |\alpha|=d}}x^\alpha
=
\sum_{\substack{\alpha\in\{0,\ldots,q-1\}^N\\ |\alpha|=d}}
\prod_{k=1}^N x_k^{\alpha_k}.
\]
Consider the product
\[
\prod_{k=1}^N\bigg(\sum_{a=0}^{q-1}(x_k z)^a\bigg).
\]
Expanding it, each choice of exponents
$\alpha=(\alpha_1,\ldots,\alpha_N)\in\{0,\ldots,q-1\}^N$ contributes
\[
\prod_{k=1}^N (x_k z)^{\alpha_k}
=
x^\alpha z^{|\alpha|}.
\]
Hence the coefficient of $z^d$ in this product is precisely
\[
\sum_{\substack{\alpha\in\{0,\ldots,q-1\}^N\\ |\alpha|=d}}x^\alpha
=
K_{\mathcal P_d(C_q^N)}(x),
\]
which proves the first identity.

For the second identity, for each $k$ we use the finite geometric sum
\[
\sum_{a=0}^{q-1}(x_k z)^a
=
\frac{1-(x_k z)^q}{1-x_k z}.
\]
Since $x_k\in C_q$, we have $x_k^q=1$, and therefore $(x_k z)^q=z^q$.
Multiplying over $k=1,\ldots,N$ gives
\[
\prod_{k=1}^N\bigg(\sum_{a=0}^{q-1}(x_k z)^a\bigg)
=
\prod_{k=1}^N\frac{1-z^q}{1-x_k z}
=
(1-z^q)^N\prod_{k=1}^N(1-x_k z)^{-1}.
\]
This completes the proof.
\end{proof}

\subsection{CLT-scale coefficient reduction}

We begin with a simple observation concerning moments of roots of unity.

\begin{lemma}\label{lem:roots-moments}
Let $X$ be uniformly distributed on
$C_q=\{\omega^j\colon j=0,\ldots,q-1\}$, where
$\omega=e^{2\pi i/q}$ and $q\ge3$. Then
\[
\mathbb E[X]=0
\qquad\text{and}\qquad
\mathbb E[X^2]=0.
\]
\end{lemma}

\begin{proof}
Since $\omega\neq1$ and $\omega^q=1$, the first moment is given by
\[
\mathbb E[X]
=
\frac1q\sum_{j=0}^{q-1}\omega^j
=
\frac{1-\omega^q}{q(1-\omega)}
=
0.
\]
For the second moment,
\[
\mathbb E[X^2]
=
\frac1q\sum_{j=0}^{q-1}\omega^{2j}.
\]
If $q$ is odd, multiplication by $2$ induces a permutation of
$\mathbb Z/q\mathbb Z$, and therefore
\[
\sum_{j=0}^{q-1}\omega^{2j}
=
\sum_{j=0}^{q-1}\omega^j
=
0.
\]
If $q$ is even, write $q=2r$ with $r\ge2$ and set
$v=\omega^2=e^{2\pi i/r}$. Then $v\neq1$ and $v^r=1$, so
\[
\sum_{j=0}^{q-1}\omega^{2j}
=
\sum_{j=0}^{q-1}v^j
=
0.
\]
In both cases, $\mathbb E[X^2]=0$.
\end{proof}

The restriction $q\ge3$ is essential here. When $q=2$, one has
$\omega=-1$, and hence $X^2\equiv1$, so $\mathbb E[X^2]=1$.

\medskip

Before continuing, we recall the probabilistic Landau notation $O_{\mathbb P}$
and $o_{\mathbb P}$. For a positive deterministic sequence $(a_N)$, we write
$X_N=O_{\mathbb P}(a_N)$ if, for every $\varepsilon>0$, there exists
$M<\infty$ such that
\[
\sup_{N\in\mathbb N}
\mathbb P\big(|X_N|>M a_N\big)
\le
\varepsilon.
\]
We write $X_N=o_{\mathbb P}(a_N)$ if
\[
\frac{X_N}{a_N}
\overset{\mathbb P}{\longrightarrow}
0,
\]
equivalently, if for every $\delta>0$,
\[
\lim_{N\to\infty}\mathbb P(|X_N|>\delta a_N)=0.
\]
We shall use the standard rules
\[
O_{\mathbb P}(a_N)\,O_{\mathbb P}(b_N)=O_{\mathbb P}(a_Nb_N),
\qquad
o_{\mathbb P}(a_N)\,O_{\mathbb P}(b_N)=o_{\mathbb P}(a_Nb_N),
\]
and the analogous rule for sums. When bounds are uniform in $t$ over a compact
set, the constants may depend on this compact set, but never on $N$.

Let $X=(X_1,\ldots,X_N)$ be uniformly distributed on $C_q^N$, so that the
coordinates $X_k$ are independent and uniformly distributed on $C_q$. Set
\[
S_N:=\sum_{k=1}^N X_k.
\]

\begin{proposition}\label{prop:CLTreduction}
Let $q\ge3$. Then, for each fixed degree $d\ge1$,
\[
\frac{K_{\mathcal P_d(C_q^N)}(X)}
     {\sqrt{\dim \mathcal P_d(C_q^N)}}
=
\frac{1}{\sqrt{d!}}
\bigg(\frac{S_N}{\sqrt N}\bigg)^{d}
+
o_{\mathbb P}(1),
\qquad N\to\infty.
\]
\end{proposition}

\begin{proof}
We divide the proof into several steps.

By Proposition~\ref{prop:GF-slice}, for every $z\in\C$ with $|z|<1$,
\begin{equation}\label{eq:homogeneous-kernel-gf}
\sum_{m\ge0} K_{\mathcal P_m(C_q^N)}(X)\,z^m
=
(1-z^q)^N\prod_{k=1}^N(1-X_kz)^{-1}.
\end{equation}
Fix $T>0$ and put
\[
z=\frac{t}{\sqrt N},
\qquad |t|\le T.
\]
For all sufficiently large $N$, we have $|X_kz|\le T/\sqrt N<1$. Hence each
factor $(1-X_kz)^{-1}$ is analytic on the disk $|t|\le T$, and the logarithmic
expansions below are valid uniformly on this disk.

For $|w|<1$,
\[
\log(1-w)^{-1}=\sum_{r\ge1}\frac{w^r}{r}.
\]
With $w=X_kz$ and $z=t/\sqrt N$, this series is absolutely and uniformly
convergent for $|t|\le T$ and all large $N$. Therefore
\[
\log\prod_{k=1}^N(1-X_kz)^{-1}
=
\sum_{k=1}^N\sum_{r\ge1}\frac{(X_kz)^r}{r}
=
\sum_{r\ge1}\frac{t^r}{rN^{r/2}}\sum_{k=1}^N X_k^r.
\]
Writing the linear part explicitly, we obtain
\[
\log\prod_{k=1}^N(1-X_kz)^{-1}
=
t\,\frac{S_N}{\sqrt N}+R_N(t),
\]
where
\[
R_N(t)
=
\frac{t^2}{2N}\sum_{k=1}^N X_k^2
+
\sum_{r\ge3}\frac{t^r}{rN^{r/2}}\sum_{k=1}^N X_k^r.
\]

We next record the bounds
\begin{equation}\label{eq:SN-OP}
\frac{S_N}{\sqrt N}=O_{\mathbb P}(1),
\end{equation}
and
\begin{equation}\label{eq:RN-small}
\sup_{|t|\le T}|R_N(t)|=o_{\mathbb P}(1).
\end{equation}

Indeed, since $|X_k|=1$ and $\mathbb E[X_k]=0$, independence gives
\[
\mathbb E|S_N|^2
=
\sum_{i=1}^N\mathbb E|X_i|^2
+
\sum_{i\ne j}\mathbb E[X_i]\mathbb E[\overline{X_j}]
=
N.
\]
Thus Chebyshev's inequality implies \eqref{eq:SN-OP}.

For the quadratic part of $R_N$, Lemma~\ref{lem:roots-moments} gives
$\mathbb E[X_k^2]=0$, and hence
\[
\mathbb E\bigg|
\frac1N\sum_{k=1}^N X_k^2
\bigg|^2
=
\frac1N.
\]
Therefore
\[
\frac1N\sum_{k=1}^N X_k^2
=
O_{\mathbb P}(N^{-1/2}),
\]
and consequently
\[
\sup_{|t|\le T}
\bigg|
\frac{t^2}{2N}\sum_{k=1}^N X_k^2
\bigg|
=
O_{\mathbb P}(N^{-1/2}).
\]
For the higher-order part, we use the deterministic estimate
$\big|\sum_{k=1}^N X_k^r\big|\le N$. Hence, uniformly for $|t|\le T$,
\[
\sum_{r\ge3}\frac{|t|^r}{rN^{r/2}}
\bigg|\sum_{k=1}^N X_k^r\bigg|
\le
\sum_{r\ge3}\frac{T^r}{r}N^{1-r/2}
=
O(N^{-1/2}).
\]
Combining the last two estimates proves \eqref{eq:RN-small}.

It follows that, uniformly for $|t|\le T$,
\[
\log\prod_{k=1}^N(1-X_kz)^{-1}
=
t\,\frac{S_N}{\sqrt N}
+
o_{\mathbb P}(1),
\qquad z=\frac{t}{\sqrt N}.
\]

We now control the prefactor $(1-z^q)^N$. Since $q\ge3$ and $z=t/\sqrt N$,
\[
N\log(1-z^q)
=
-Nz^q+O(N|z|^{2q})
=
-t^qN^{1-q/2}+O(N^{1-q}|t|^{2q})
=
o(1),
\]
uniformly for $|t|\le T$. Hence
\[
(1-z^q)^N=1+o(1)
\]
uniformly for $|t|\le T$.

Returning to \eqref{eq:homogeneous-kernel-gf}, exponentiating the logarithmic
estimate and multiplying by the prefactor, we obtain
\[
\sum_{m\ge0}
\frac{K_{\mathcal P_m(C_q^N)}(X)}{N^{m/2}}\,t^m
=
\exp\bigg(t\,\frac{S_N}{\sqrt N}\bigg)
\bigl(1+r_N(t)\bigr),
\]
where
\[
\sup_{|t|\le T}|r_N(t)|=o_{\mathbb P}(1).
\]

We now extract the coefficient of $t^d$. Fix $0<\rho<T$. By Cauchy's integral
formula,
\[
\frac{K_{\mathcal P_d(C_q^N)}(X)}{N^{d/2}}
=
\frac{1}{2\pi i}
\oint_{|t|=\rho}
\frac{1}{t^{d+1}}
\sum_{m\ge0}
\frac{K_{\mathcal P_m(C_q^N)}(X)}{N^{m/2}}\,t^m
\,dt.
\]
Thus
\[
\frac{K_{\mathcal P_d(C_q^N)}(X)}{N^{d/2}}
=
M_N+E_N,
\]
where
\[
M_N
=
\frac{1}{2\pi i}
\oint_{|t|=\rho}
\frac{\exp(tS_N/\sqrt N)}{t^{d+1}}\,dt
\]
and
\[
E_N
=
\frac{1}{2\pi i}
\oint_{|t|=\rho}
\frac{\exp(tS_N/\sqrt N)r_N(t)}{t^{d+1}}\,dt.
\]
Expanding the exponential and applying Cauchy's formula gives
\[
M_N
=
\frac1{d!}
\bigg(\frac{S_N}{\sqrt N}\bigg)^d.
\]
Moreover,
\[
|E_N|
\le
\rho^{-d}
\exp\bigg(\rho\,\frac{|S_N|}{\sqrt N}\bigg)
\sup_{|t|\le T}|r_N(t)|.
\]
By \eqref{eq:SN-OP}, $S_N/\sqrt N=O_{\mathbb P}(1)$, and hence
\[
\exp\bigg(\rho\,\frac{|S_N|}{\sqrt N}\bigg)=O_{\mathbb P}(1).
\]
Since $\sup_{|t|\le T}|r_N(t)|=o_{\mathbb P}(1)$, it follows that
$E_N=o_{\mathbb P}(1)$. Therefore
\[
\frac{K_{\mathcal P_d(C_q^N)}(X)}{N^{d/2}}
=
\frac1{d!}
\bigg(\frac{S_N}{\sqrt N}\bigg)^d
+
o_{\mathbb P}(1).
\]

Finally, by Corollary~\ref{dimoo},
\[
\sqrt{\dim\mathcal P_d(C_q^N)}
\sim
\frac{N^{d/2}}{\sqrt{d!}},
\qquad N\to\infty.
\]
Dividing by this asymptotic relation gives
\[
\frac{K_{\mathcal P_d(C_q^N)}(X)}
     {\sqrt{\dim \mathcal P_d(C_q^N)}}
=
\frac{1}{\sqrt{d!}}
\bigg(\frac{S_N}{\sqrt N}\bigg)^d
+
o_{\mathbb P}(1),
\]
as claimed.
\end{proof}

\subsection{Proof of Theorem~\ref{thm:asymptotic-projection-Pd} for $\mathcal P_d(C_q^N)$}

In this subsection we prove the homogeneous part of
Theorem~\ref{thm:asymptotic-projection-Pd} for $q\ge3$. The case $q=2$ is
covered by Theorem~\ref{thm: main spherical}, since in the Boolean case the
homogeneous, spherical, and tetrahedral spaces of degree $d$ coincide.

Let $X=(X_1,\ldots,X_N)$ be uniformly distributed on $C_q^N$, so that the
coordinates $X_k$ are independent and uniformly distributed on $C_q$. By the
general integral formula for projection constants on finite Abelian groups,
\eqref{hamming-int}, we have
\[
\boldsymbol{\lambda}\big(\mathcal P_d(C_q^N)\big)
=
\mathbb E\big|K_{\mathcal P_d(C_q^N)}(X)\big|.
\]
Equivalently,
\[
\frac{\boldsymbol{\lambda}\big(\mathcal P_d(C_q^N)\big)}
     {\sqrt{\dim \mathcal P_d(C_q^N)}}
=
\mathbb E\bigg[
\bigg|
\frac{K_{\mathcal P_d(C_q^N)}(X)}
     {\sqrt{\dim \mathcal P_d(C_q^N)}}
\bigg|
\bigg].
\]

\begin{proof}[Proof of Theorem~\ref{thm:asymptotic-projection-Pd}, homogeneous case for $q\ge3$]
For every $N$, by orthonormality of the degree--$d$ characters,
\[
\mathbb E\bigg[
\bigg|
\frac{K_{\mathcal P_d(C_q^N)}(X)}
     {\sqrt{\dim \mathcal P_d(C_q^N)}}
\bigg|^2
\bigg]
=1.
\]
Indeed,
\[
\mathbb E\big|K_{\mathcal P_d(C_q^N)}(X)\big|^2
=
\dim \mathcal P_d(C_q^N).
\]
Thus the family
\[
\left\{
\frac{K_{\mathcal P_d(C_q^N)}(X)}
     {\sqrt{\dim \mathcal P_d(C_q^N)}}
\right\}_{N}
\]
is uniformly integrable.

By Proposition~\ref{prop:CLTreduction}, for each fixed $d\ge1$,
\[
\frac{K_{\mathcal P_d(C_q^N)}(X)}
     {\sqrt{\dim \mathcal P_d(C_q^N)}}
=
\frac{1}{\sqrt{d!}}
\bigg(\frac{S_N}{\sqrt N}\bigg)^d
+
o_{\mathbb P}(1),
\qquad N\to\infty,
\]
where $S_N=\sum_{k=1}^N X_k$.

Identifying $\C$ with $\mathbb R^2$, the multivariate central limit theorem
applied to the independent random variables uniformly distributed on the
$q$-th roots of unity, with $q\ge3$, gives
\[
\frac{S_N}{\sqrt N}
\overset{D}{\longrightarrow}
\mathbf G,
\qquad N\to\infty,
\]
where $\mathbf G$ is a standard circular complex Gaussian random variable,
that is,
\[
\mathbf G=\frac{G_1+iG_2}{\sqrt2},
\qquad
G_1,G_2\ \text{independent and } N(0,1)\text{-distributed}.
\]
Indeed, Lemma~\ref{lem:roots-moments} implies that the covariance matrix of
$(\operatorname{Re}X_1,\operatorname{Im}X_1)$ is $\frac12 I_2$.  To see this, we
write $X_1=A+iB$. Since $\mathbb E (X_1)=0$, we have
$\mathbb E (A)=\mathbb E (B)=0$. Moreover, $\mathbb E (X_1^2)=0$ gives
$\mathbb E(A^2-B^2)=0$ and $\mathbb E(AB)=0$. Finally, since
$|X_1|^2=A^2+B^2=1$, we get
$\mathbb E (A^2)+\mathbb E (B^2)=1$. Hence
$\mathbb E (A^2)=\mathbb E (B^2)=1/2$ and $\mathbb E(AB)=0$, as claimed.

By the continuous mapping theorem and Slutsky's theorem, we obtain
\[
\frac{K_{\mathcal P_d(C_q^N)}(X)}
     {\sqrt{\dim \mathcal P_d(C_q^N)}}
\overset{D}{\longrightarrow}
\frac{\mathbf G^d}{\sqrt{d!}},
\qquad N\to\infty.
\]
Since the normalized kernels are uniformly integrable, convergence in
distribution implies convergence of the first absolute moments. Hence
\[
\lim_{N\to\infty}
\mathbb E\bigg[
\bigg|
\frac{K_{\mathcal P_d(C_q^N)}(X)}
     {\sqrt{\dim \mathcal P_d(C_q^N)}}
\bigg|
\bigg]
=
\mathbb E\bigg|\frac{\mathbf G^d}{\sqrt{d!}}\bigg|
=
\frac{\mathbb E|\mathbf G|^d}{\sqrt{d!}}.
\]
Since $|\mathbf G|^2$ has the exponential distribution with mean $1$, we have
\[
\mathbb E|\mathbf G|^d
=
\Gamma\bigg(1+\frac d2\bigg).
\]
Therefore
\[
\lim_{N\to\infty}
\frac{\boldsymbol{\lambda}\big(\mathcal P_d(C_q^N)\big)}
     {\sqrt{\dim \mathcal P_d(C_q^N)}}
=
\frac{\Gamma\big(1+\frac d2\big)}{\sqrt{d!}},
\]
which proves the homogeneous part of Theorem~\ref{thm:asymptotic-projection-Pd}
for $q\ge3$.
\end{proof}

\subsection{Proof of Theorem~\ref{thm:asymptotic-projection-Pd} for $\mathcal T_d(C_q^N)$}

In view of the following comparison lemma, the tetrahedral case follows from
the homogeneous case.

\begin{lemma}\label{lem:Pd-Td-projection-comparison}
Fix $q\ge2$ and $d\ge1$. Then, as $N\to\infty$,
\[
\frac{\boldsymbol{\lambda}\big(\mathcal P_d(C_q^N)\big)}
     {\sqrt{\dim \mathcal P_d(C_q^N)}}
-
\frac{\boldsymbol{\lambda}\big(\mathcal T_d(C_q^N)\big)}
     {\sqrt{\dim \mathcal T_d(C_q^N)}}
\longrightarrow 0.
\]
\end{lemma}

\begin{proof}
If $q=2$, then every $\bm m\in\{0,1\}^N$ with $|\bm m|=d$ satisfies
$|\operatorname{supp}(\bm m)|=d$. Hence
\[
\mathcal P_d(C_2^N)=\mathcal T_d(C_2^N),
\]
and there is nothing to prove.

We may therefore assume that $q\ge3$. Let
\[
I_P
:=
\big\{\bm m\in\{0,\ldots,q-1\}^N\colon |\bm m|=d\big\},
\qquad
I_T
:=
\big\{\bm m\in\{0,1\}^N\colon |\bm m|=d\big\}.
\]
Thus $I_T\subset I_P$. Define
\[
R_{N,d,q}(x)
:=
K_{\mathcal P_d(C_q^N)}(x)-K_{\mathcal T_d(C_q^N)}(x)
=
\sum_{\bm m\in I_P\setminus I_T}\chi_{\bm m}(x).
\]
If $\bm m\in I_P\setminus I_T$, then $|\bm m|=d$ and at least one nonzero
coordinate of $\bm m$ is larger than $1$. Hence
\[
|\operatorname{supp}(\bm m)|\le d-1.
\]
Consequently,
\[
I_P\setminus I_T
\subset
\bigcup_{r=1}^{d-1}
\big\{
\bm m\in\{0,\ldots,q-1\}^N\colon
|\bm m|=d,\ |\operatorname{supp}(\bm m)|=r
\big\}.
\]
For fixed $r$, the support can be chosen in $\binom Nr$ ways, and once the
support is fixed the number of possible exponent vectors is bounded by a
constant depending only on $d$ and $q$. Therefore
\[
|I_P\setminus I_T|
\le
C_{d,q}\sum_{r=1}^{d-1}\binom Nr
=
O_{d,q}(N^{d-1}).
\]
By orthonormality of the characters on $C_q^N$,
\[
\mathbb E\big|R_{N,d,q}(X)\big|^2
=
|I_P\setminus I_T|
=
O_{d,q}(N^{d-1}),
\]
where $X$ is uniformly distributed on $C_q^N$. Hence, by Cauchy's inequality,
\[
\mathbb E|R_{N,d,q}(X)|
\le
\big(\mathbb E|R_{N,d,q}(X)|^2\big)^{1/2}
=
O_{d,q}(N^{(d-1)/2}).
\]
By \eqref{hamming-int},
\[
\boldsymbol{\lambda}\big(\mathcal P_d(C_q^N)\big)
=
\mathbb E\big|K_{\mathcal P_d(C_q^N)}(X)\big|,
\qquad
\boldsymbol{\lambda}\big(\mathcal T_d(C_q^N)\big)
=
\mathbb E\big|K_{\mathcal T_d(C_q^N)}(X)\big|.
\]
Therefore the triangle inequality gives
\[
\big|
\boldsymbol{\lambda}\big(\mathcal P_d(C_q^N)\big)
-
\boldsymbol{\lambda}\big(\mathcal T_d(C_q^N)\big)
\big|
\le
\mathbb E|R_{N,d,q}(X)|
=
O_{d,q}(N^{(d-1)/2}).
\]
On the other hand,
\[
\dim\mathcal T_d(C_q^N)=\binom Nd\sim\frac{N^d}{d!},
\]
whereas Corollary~\ref{dimoo} gives
\[
\dim\mathcal P_d(C_q^N)\sim\frac{N^d}{d!}.
\]
In particular,
\[
\frac{\sqrt{\dim\mathcal T_d(C_q^N)}}
     {\sqrt{\dim\mathcal P_d(C_q^N)}}
\longrightarrow 1.
\]
It follows that
\[
\frac{
\big|
\boldsymbol{\lambda}\big(\mathcal P_d(C_q^N)\big)
-
\boldsymbol{\lambda}\big(\mathcal T_d(C_q^N)\big)
\big|
}{
\sqrt{\dim\mathcal P_d(C_q^N)}
}
=
O_{d,q}(N^{-1/2})
\longrightarrow 0.
\]

Moreover, again by Cauchy's inequality and orthonormality,
\[
\boldsymbol{\lambda}\big(\mathcal T_d(C_q^N)\big) \le \sqrt{\dim\mathcal T_d(C_q^N)}.
\]
Thus
\[
\boldsymbol{\lambda}\big(\mathcal T_d(C_q^N)\big)\Bigg| \frac{1}{\sqrt{\dim\mathcal P_d(C_q^N)}}
-
\frac{1}{\sqrt{\dim\mathcal T_d(C_q^N)}}\Bigg| \le
\Bigg|\frac{\sqrt{\dim\mathcal T_d(C_q^N)}}{\sqrt{\dim\mathcal P_d(C_q^N)}}
-1\Bigg|=o(1).
\]
Combining the last two estimates yields
\[
\Bigg|
\frac{\boldsymbol{\lambda}\big(\mathcal P_d(C_q^N)\big)} {\sqrt{\dim\mathcal P_d(C_q^N)}}
- \frac{\boldsymbol{\lambda}\big(\mathcal T_d(C_q^N)\big)}
{\sqrt{\dim\mathcal T_d(C_q^N)}}\Bigg|\longrightarrow 0,
\]
as claimed.
\end{proof}

\subsection{The top-layer principle for total-degree spaces}

We conclude by showing that the same top-layer principle remains valid in the homogeneous and tetrahedral settings. More precisely, for spaces of degree at most $d$, the asymptotic behaviour of the normalized projection constants is entirely determined by the component of degree exactly $d$.

\begin{proposition}\label{prop:P-T-le-d}
Fix $q\ge2$ and $d\ge1$. Then
\[
\lim_{N\to\infty}
\frac{\boldsymbol{\lambda}\big(\mathcal P_{\le d}(C_q^N)\big)}
     {\sqrt{\dim \mathcal P_{\le d}(C_q^N)}}
=
\lim_{N\to\infty}
\frac{\boldsymbol{\lambda}\big(\mathcal T_{\le d}(C_q^N)\big)}
     {\sqrt{\dim \mathcal T_{\le d}(C_q^N)}}
=
\begin{cases}
\displaystyle
\frac{1}{\sqrt{d!}}\,\mathbb E\big|\mathrm{He}_d(Z)\big|,
& q=2,\\[3ex]
\displaystyle
\frac{\Gamma\big(1+\frac d2\big)}{\sqrt{d!}},
& q\ge3,
\end{cases}
\]
where $Z$ is a standard Gaussian random variable.
\end{proposition}

\begin{proof}
Let $X=(X_1,\ldots,X_N)$ be uniformly distributed on $C_q^N$. By
Theorem~\ref{hamming-int},
\[
\boldsymbol{\lambda}\big(\mathcal P_{\le d}(C_q^N)\big)
=
\mathbb E\big|K_{\mathcal P_{\le d}(C_q^N)}(X)\big|,
\]
where
\[
K_{\mathcal P_{\le d}(C_q^N)}(x)
:=
\sum_{k=0}^d K_{\mathcal P_k(C_q^N)}(x)
\qquad
\text{and}
\qquad
K_{\mathcal P_k(C_q^N)}(x)
:=
\sum_{\substack{\bm m\in\{0,\ldots,q-1\}^N\\ |\bm m|=k}}
\chi_{\bm m}(x).
\]
The characters belonging to different homogeneous degrees are distinct, hence
orthonormal in $L^2(C_q^N)$. Therefore
\[
\mathbb E\big|K_{\mathcal P_k(C_q^N)}(X)\big|^2
=
\dim \mathcal P_k(C_q^N).
\]
For every fixed $k\ge1$, Corollary~\ref{dimoo} gives
\[
\dim \mathcal P_k(C_q^N)
\sim
\frac{N^k}{k!},
\qquad N\to\infty,
\]
while $\dim\mathcal P_0(C_q^N)=1$. Consequently,
\[
\dim \mathcal P_{\le d}(C_q^N)
=
\sum_{k=0}^d \dim \mathcal P_k(C_q^N)
\sim
\dim \mathcal P_d(C_q^N)
\sim
\frac{N^d}{d!},
\]
and in particular
\[
\frac{\dim \mathcal P_d(C_q^N)}
     {\dim \mathcal P_{\le d}(C_q^N)}
\longrightarrow 1.
\]

We compare the normalized kernels
\[
\frac{K_{\mathcal P_{\le d}(C_q^N)}}
     {\sqrt{\dim \mathcal P_{\le d}(C_q^N)}}
\quad\text{and}\quad
\frac{K_{\mathcal P_d(C_q^N)}}
     {\sqrt{\dim \mathcal P_d(C_q^N)}}.
\]
By the triangle inequality,
\[
\begin{aligned}
&
\mathbb E\bigg|
\frac{K_{\mathcal P_{\le d}(C_q^N)}(X)}
     {\sqrt{\dim \mathcal P_{\le d}(C_q^N)}}
-
\frac{K_{\mathcal P_d(C_q^N)}(X)}
     {\sqrt{\dim \mathcal P_d(C_q^N)}}
\bigg|
\\
&\le
\bigg|
\sqrt{
\frac{\dim \mathcal P_d(C_q^N)}
     {\dim \mathcal P_{\le d}(C_q^N)}
}
-1
\bigg|
\mathbb E\bigg|
\frac{K_{\mathcal P_d(C_q^N)}(X)}
     {\sqrt{\dim \mathcal P_d(C_q^N)}}
\bigg|
\quad+
\sum_{k=0}^{d-1}
\frac{\mathbb E\big|K_{\mathcal P_k(C_q^N)}(X)\big|}
     {\sqrt{\dim \mathcal P_{\le d}(C_q^N)}}.
\end{aligned}
\]
The first term tends to $0$, because the expectation appearing there is at
most $1$ by Cauchy--Schwarz and orthogonality. For the remaining terms,
Cauchy--Schwarz and orthogonality give
\[
\mathbb E\big|K_{\mathcal P_k(C_q^N)}(X)\big|
\le
\sqrt{\dim \mathcal P_k(C_q^N)}.
\]
Consequently, for $0\le k<d$,
\[
\frac{\mathbb E\big|K_{\mathcal P_k(C_q^N)}(X)\big|}
     {\sqrt{\dim \mathcal P_{\le d}(C_q^N)}}
\le
\bigg(
\frac{\dim \mathcal P_k(C_q^N)}
     {\dim \mathcal P_{\le d}(C_q^N)}
\bigg)^{1/2}
=
O(N^{(k-d)/2})
\longrightarrow 0.
\]
Thus
\[
\mathbb E\bigg|
\frac{K_{\mathcal P_{\le d}(C_q^N)}(X)}
     {\sqrt{\dim \mathcal P_{\le d}(C_q^N)}}
-
\frac{K_{\mathcal P_d(C_q^N)}(X)}
     {\sqrt{\dim \mathcal P_d(C_q^N)}}
\bigg|
\longrightarrow 0.
\]
It follows that
\[
\lim_{N\to\infty}
\mathbb E\bigg|
\frac{K_{\mathcal P_{\le d}(C_q^N)}(X)}
     {\sqrt{\dim \mathcal P_{\le d}(C_q^N)}}
\bigg|
=
\lim_{N\to\infty}
\mathbb E\bigg|
\frac{K_{\mathcal P_d(C_q^N)}(X)}
     {\sqrt{\dim \mathcal P_d(C_q^N)}}
\bigg|.
\]
The homogeneous part of Theorem~\ref{thm:asymptotic-projection-Pd} therefore
gives
\[
\lim_{N\to\infty}
\frac{\boldsymbol{\lambda}\big(\mathcal P_{\le d}(C_q^N)\big)}
     {\sqrt{\dim \mathcal P_{\le d}(C_q^N)}}
=
\begin{cases}
\displaystyle
\frac{1}{\sqrt{d!}}\,\mathbb E\big|\mathrm{He}_d(Z)\big|,
& q=2,\\[3ex]
\displaystyle
\frac{\Gamma\big(1+\frac d2\big)}{\sqrt{d!}},
& q\ge3.
\end{cases}
\]

We now treat $\mathcal T_{\le d}(C_q^N)$. For $0\le k\le d$, set
\[
K_{\mathcal T_k(C_q^N)}(x)
:=
\sum_{\substack{\bm m\in\{0,1\}^N\\ |\bm m|=k}}
\chi_{\bm m}(x)
\qquad \text{and} \qquad
K_{\mathcal T_{\le d}(C_q^N)}(x)
=
\sum_{k=0}^d K_{\mathcal T_k(C_q^N)}(x).
\]
Again, by orthogonality,
\[
\mathbb E\big|K_{\mathcal T_k(C_q^N)}(X)\big|^2
=
\dim \mathcal T_k(C_q^N).
\]
Moreover,
\[
\dim \mathcal T_{\le d}(C_q^N)
=
\sum_{k=0}^d \binom Nk
\sim
\dim \mathcal T_d(C_q^N)
=
\binom Nd
\sim
\frac{N^d}{d!}.
\]
Repeating the argument used for $\mathcal P_{\le d}(C_q^N)$, we obtain
\[
\mathbb E\Bigg|
\frac{K_{\mathcal T_{\le d}(C_q^N)}(X)}
     {\sqrt{\dim \mathcal T_{\le d}(C_q^N)}}
-
\frac{K_{\mathcal T_d(C_q^N)}(X)}
     {\sqrt{\dim \mathcal T_d(C_q^N)}}
\Bigg|
\longrightarrow 0.
\]
Hence
\[
\lim_{N\to\infty}
\mathbb E\Bigg|
\frac{K_{\mathcal T_{\le d}(C_q^N)}(X)}
     {\sqrt{\dim \mathcal T_{\le d}(C_q^N)}}
\Bigg|
=
\lim_{N\to\infty}
\mathbb E\Bigg|
\frac{K_{\mathcal T_d(C_q^N)}(X)}
     {\sqrt{\dim \mathcal T_d(C_q^N)}}
\Bigg|.
\]
The tetrahedral part of Theorem~\ref{thm:asymptotic-projection-Pd}, together
with \eqref{hamming-int}, gives the desired conclusion for
$\mathcal T_{\le d}(C_q^N)$.
\end{proof}

% \bibliographystyle{abbrv}
% \bibliography{biblio}

\end{document}